\documentclass[a4paper,reqno,11pt]{amsart}
\usepackage[a4paper,margin=1in]{geometry}
\usepackage{graphicx}
\usepackage{amsmath,amssymb,amsthm}
\usepackage{subcaption}

\theoremstyle{plain}
\newtheorem{theorem}{Theorem}[section]
\newtheorem{lemma}[theorem]{Lemma}
\newtheorem{proposition}[theorem]{Proposition}

\newtheorem{corollary}[theorem]{Corollary}
\theoremstyle{definition}
\newtheorem{remark}[theorem]{Remark}
\numberwithin{equation}{section}

\newcommand{\R}{\mathbb{R}}

\title[On negative eigenvalues]{On negative eigenvalues of the spectral problem for water waves of highest amplitude}

\author{Vladimir Kozlov$^1$, Evgeniy Lokharu$^1$}

\begin{document}
	
\begin{abstract} We consider a spectral problem associated with steady water waves of extreme form on the free surface of a rotational flow. It is proved that the spectrum of this problem contains arbitrary large negative eigenvalues and they are simple. Moreover, the asymptotics of such eigenvalues is obtained.

\end{abstract}

\maketitle

\section{Introduction} \label{s:introduction}

Extreme waves (or, equivalently, ``waves of greatest height'' according to Stokes) are remarkable objects in the mathematical theory of water waves. These are normally large-amplitude travelling waves with sharp crests of included angle $120^\circ$, see Figure 1A. Extreme waves were conjectured by Stokes already in 1880s. In \cite{Stokes49} Stokes considered periodic solutions to the water wave problem with a fixed wavelength and assumed that such waves can be parametrized by the wave height $\sup_{x \in \R} \eta(x) - \inf_{x \in \R} \eta(x)$, where $\eta$ is the surface profile. Later he conjectured in \cite{Stokes80} that the family of periodic waves contains the ``\textit{wave of greatest height}'' with surface stagnation, distinguished by sharp crests of included angle $120^\circ$. Stokes also argued that the stagnation by itself forces the surface profile to have a sharp crest of included angle $120^\circ$. This property is known as the \textit{Stokes conjecture} about waves of greatest height, which stimulated the development of the theory for many years. It is reasonable to divide the Stokes conjecture into two parts: (i) there exists a travelling solution of the water wave problem that enjoys stagnation at every crest; (ii) Every solution from (i) with surface profile $\eta$ must satisfy
	\[
	\lim_{x \to x_0\pm} \eta_x(x) = \mp \frac{1}{\sqrt{3}}
	\]
at every stagnation point $(x_0,\eta(x_0))$; this corresponds to the included angle $120^\circ$ as illustrated in Figure 1.
Both statements were complicated problems for the time because solutions with surface stagnation points are large-amplitude waves and could not be analysed by the classical perturbation methods. The first existence of Stokes waves that are close to the stagnation is due to Keady and Norbury \cite{KeadyNorbury78}, who used a global bifurcation theory for positive operators applied to the Nekrasov equation. Thus, one could think of proving (i) by passing to the limit along a sequence of waves approaching stagnation. This was done by Toland \cite{Toland1978a} in 1978 for the infinite depth case and by Amick and Toland \cite{AmickToland81a} for waves of finite depth. The second part (ii) for Stokes waves (periodic waves, symmetric around each crest and trough and monotone in between) was verified independently by Amick, Fraenkel and Toland in \cite{AmickFraenkelToland82} and by Plotnikov \cite{Plotnikov82}. Later, Plotnikov and Toland \cite{Plotnikov2004} proved the existence of irrotational waves of extreme form that are convex everywhere outside crests. The second part of the conjecture was refined by Varvaruca and Weiss in \cite{Varvaruca2011}, who proved (ii) for solutions under weak regularity assumptions and without any symmetry or monotonicity constraints. In particular, (ii) turned out to be a local property and is valid for the extreme solitary wave found in \cite{AmickToland81b}.

All previously mentioned results concerned irrotational water waves, while the case of waves with vorticity is much less studied. There is also a qualitative difference. In their study \cite{Varvaruca2012} Varvaruca and Weiss found (without proving the existence) that surface profiles near stagnation points are either Stokes corners ($120^\circ$), horizontally flat, or horizontal cusps, though it is not known if the last two options are possible. Regarding the first question (i), it was shown in \cite{Varvaruca09} that there exists a family of periodic solutions to the water wave problem with "negative" vorticity converging to an extreme wave enjoying stagnation at every crest. Unfortunately, it was not possible to show that the limiting wave is not "trivial", that is its surface is not a horizontal line. This difficulty was resolved in \cite{LokharuKoz2020} by using a different approach and extreme waves subject to (i) were found. A further analysis was made in \cite{Koz2020asymp}, where authors obtained higher-order asymptotics for the surface profile near stagnation points of type (ii). It was shown that vorticity affects the shape of an extreme wave near the stagnation point: it is convex for non-positive vorticities and concave otherwise. This observation is confirmed by several numerical studies, such as \cite{JOY2008} and \cite{Dyachenko2019a}.

Only Stokes and solitary waves were known until 1980, when Chen and Saffman \cite{Chen1980} numerically found new types of periodic waves of infinite depth bifurcating from near-extreme Stokes waves. The result was generalized by Vanden-Broeck \cite{Vanden_Broeck_1983} to the case of a finite depth. As in the infinite depth case some new "irregular" waves were found that bifurcate from regular Stokes waves. Such irregular waves have crests at different heights so that more than one crest is observed within the minimal period. A further analysis was made in \cite{VandenBroeck2017}. It was shown that there exist bifurcations of irregular waves that approach stagnation. The limiting wave has infinitely many oscillations and one sharp crest of included angle $120^\circ$. Irregular waves with an infinite period were found in \cite{VandenBroeck2013}. The only analytical study of irregular waves of infinite depth is by Buffoni, Dancer and Toland \cite{BuffoniDancerToland00b}. The authors investigated the global bifurcation continuum of waves with a fixed period $\Lambda$. The latter set is connected and contains an extreme wave in its closure. They proved that there are infinitely many points along the continuum which are either turning points or give rise to sub-harmonic bifurcations of waves whose minimal periods are integer multiples of $\Lambda$. Such new bifurcations of irregular waves occurs from Stokes waves that are close to stagnation, for which the associated spectral problem possesses a finite but arbitrary large number of negative eigenvalues.

The main subject of this paper is an analysis of the corresponding spectral problem for extreme Stokes waves in the case of finite depth and in the presence of vorticity. Here we cannot use the Nekrasov equation and the spectral problem is formulated in terms of a boundary value problem for a partial differential equation representing the first variation of the limit problem. We show that the spectrum of such problems contains negative eigenvalues with arbitrary large absolute values. We obtain also their asymptotics and simplicity of large negative eigenvalues. Our main Theorem 1.1 is formulated and proved in a more general form, where we allow for arbitrary singularities of coner type. An application for the water wave problem is given in Section \ref{SecWater}.

\begin{figure}[t!]
	\centering
	\begin{subfigure}[t]{0.5\textwidth}
		\centering
		\includegraphics[scale=0.8]{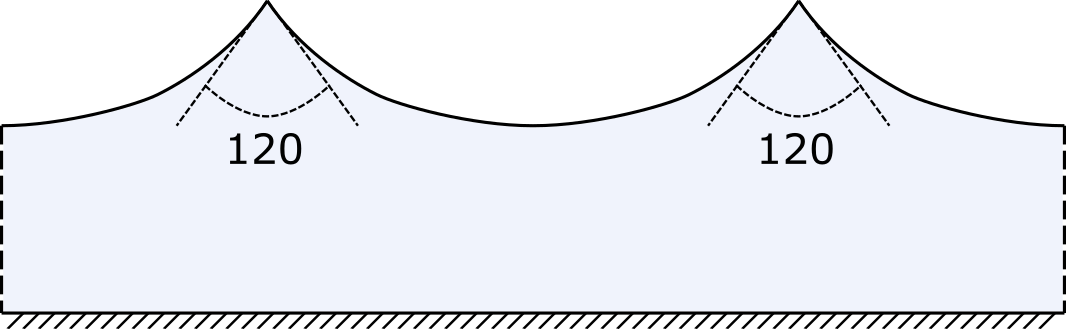}
		\caption{an extreme Stokes wave}
	\end{subfigure}%
	~ 
	\begin{subfigure}[t]{0.5\textwidth}
		\centering
		\includegraphics[scale=0.8]{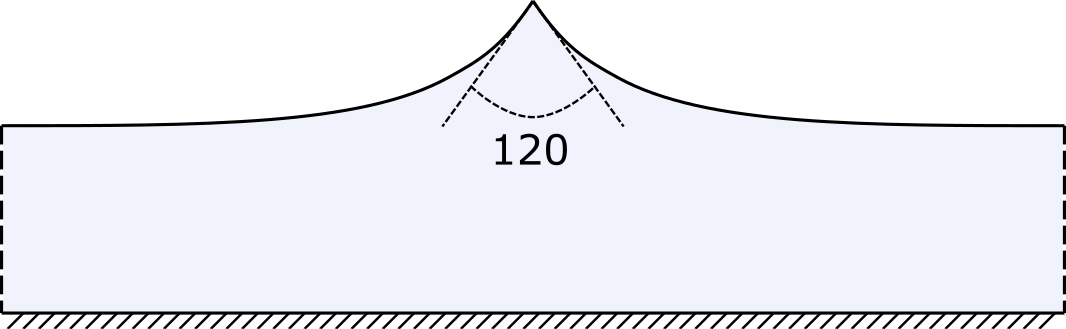} \caption{an extreme solitary wave}
	\end{subfigure}
	\caption{}
\end{figure}

\subsection{Formulation of the problem}

Let $\eta=\eta(x)$ be a positive, continuous and periodic  function on $\Bbb R$ of a period $\Lambda>0$. We assume that $\eta$ is even, i.e. $\eta(x)=\eta(-x)$, and that $\eta$ belongs to $C^{2}$  outside   the points $k\Lambda$, $k\in\Bbb Z$.
 Near the origin it has an asymptotics
\begin{equation}\label{Kk1}
\eta'(x)=-a_0+O(x^{\alpha}),\;\;\mbox{and}\;\;\eta''(x)=O(x^{\alpha-1})\;\;\mbox{for $x>0$}.
\end{equation}
Here $\alpha\in (0,1]$. Since the function $\eta$ is even the same expression with $-a_0$ replaced by $a_0$ is valid for negative $x$, and due to periodicity similar relations are true in a neighborhoods of the points $k\Lambda$.

It will be useful to introduce the angle $\alpha^*$ between the vertical line and the tangent to $\eta(x)$, $x\geq 0$, at the point $x=0$. It is defined by   $\alpha^*=\pi/2-\arctan a_0$.

Let
\begin{equation}\label{F24a}
\hat{D}=\{(x,y)\in\Bbb R^2\,:\,0<y<\eta(x)\},\;\;\hat{S}=\{(x,\eta(x))\,:\,x\in\Bbb R\},\;\;\hat{B}=\{(x,0)\,:\,x\in\Bbb R\}.
\end{equation}
Consider the following spectral  problem
\begin{eqnarray}\label{K2ax}
&&-\Delta u+\sigma(x)u=\lambda u\;\;\mbox{in $\hat{D}$}\nonumber\\
&&\partial_\nu u-r^{-1}\rho u=0\;\;\mbox{on $\hat{S}$}\nonumber\\
&&u=0\;\;\mbox{on $\hat{B}$}.
\end{eqnarray}
Here $r$ is the distance from $(x,y)$ to $(0,\eta(0))$,   $\sigma$ and $\rho$ are  $\Lambda$ periodic, even  functions, $\sigma$ is supposed to be bounded and $\rho$ is  $C^{1}$ outside the points $k\Lambda$, $k\in\Bbb Z$,  and
\begin{equation}\label{Kk1a}
\rho=\rho_0+O(x^{\alpha})\;\;\mbox{and}\;\;\frac{d\rho}{dx}=O(x^{-1+\alpha}).
\end{equation}
 It is assumed that $\rho_0>0$ and that
\begin{equation}\label{K5y}
\mu_1>1,
\end{equation}
where $\mu_1$ is the first positive root of the equation
\begin{equation}\label{K5}
\mu\tan\big(\mu\alpha^*\big)=-\rho_0.
\end{equation}
Since this root satisfies $\mu_1\alpha^*\in (\pi/2,\pi)$,
a sufficient condition for (\ref{K5y}) is $\alpha^*\leq \pi/2$.

We are looking for periodic, even functions $u$ in (\ref{K2ax}).

By our assumptions all functions $\eta$, $\sigma$, $\rho$ and $u$ are even with respect to vertical lines $x=k\Lambda/2$. If we introduce
$$
\Omega=\{(x,y)\in D\,:\,0<x<\Lambda/2\},\;\;S=\{(x,\eta(x))\,:\,0<x<\Lambda/2\},
$$
$$
B=\{(x,0)\,:\,0<x<\Lambda/2\},
$$
then the problem (\ref{K2ax}) can be reduced to the domain $\Omega$ (see Figure 2):
\begin{equation}\label{K2bb}
-\Delta u+\sigma u=\lambda u\;\;\mbox{in $\Omega$}
\end{equation}
and
\begin{eqnarray}\label{K2bbzz}
&&\partial_\nu u-r^{-1}\rho u=0\;\;\mbox{on $S$}\nonumber\\
&&u=0\;\;\mbox{on $B$}\nonumber\\
&&\partial_xu|_{x=0}=\partial_xu|_{x=\Lambda/2}=0.
\end{eqnarray}

\begin{figure}[t!]
	\centering
	\includegraphics[scale=0.8]{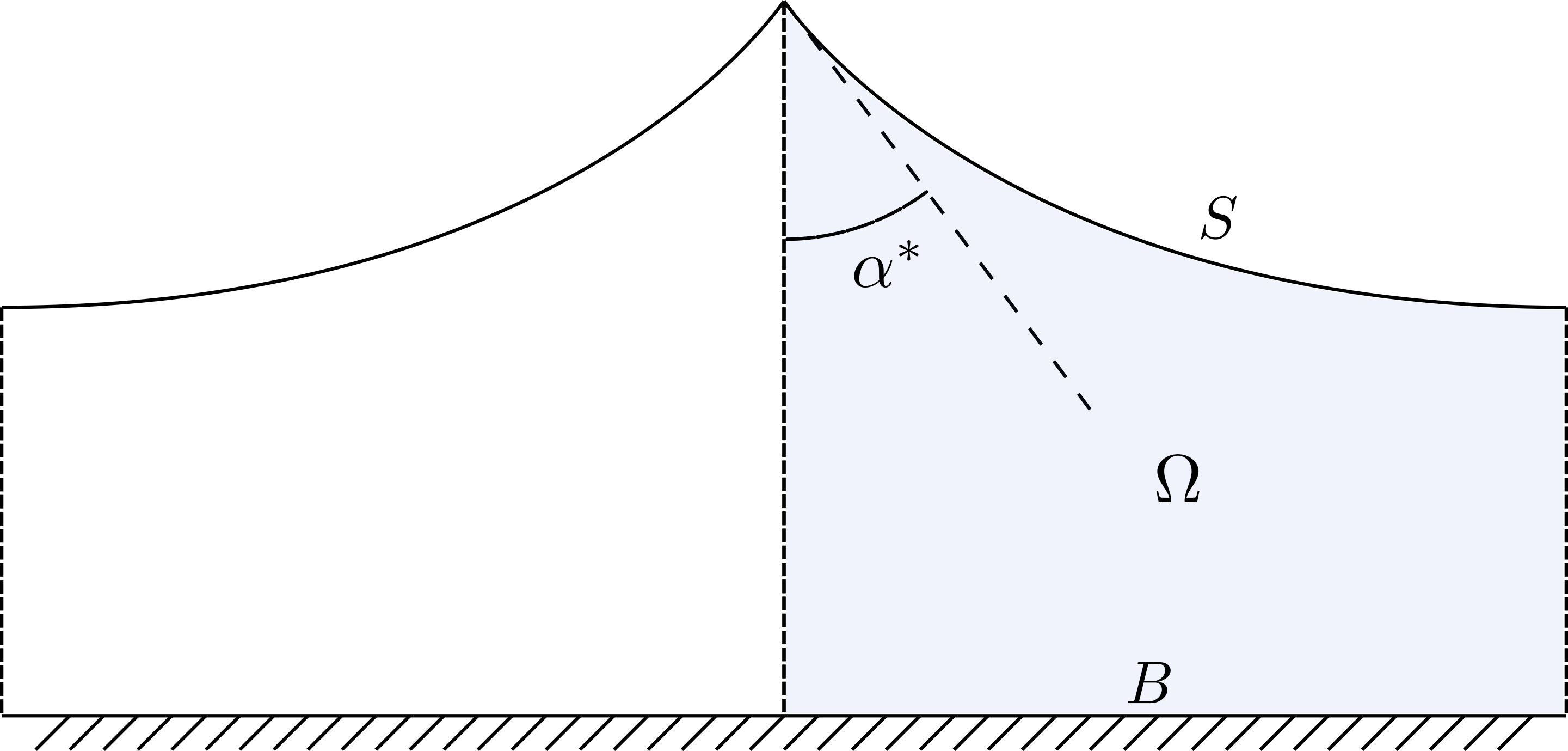}
	\caption{A sketch of the domain}
\end{figure}

Denote ${\mathcal L}=-\Delta+\sigma$ and let $V^2_\beta(\Omega)$, $\beta\in\Bbb R$, denote the space of functions $u$ defined on $\Omega$ which are subject to
$$
||u||_{V^2_\beta(\Omega)}^2:=\int_\Omega r^{2\beta}(|\nabla^2u|^2+r^{-2}|\nabla u|^2+r^{-4}|u|^2)dxdy<\infty.
$$
 The operator  ${\mathcal L}$ is symmetric on functions in $V^2_0(\Omega)$ satisfying (\ref{K2bbzz}). There are many one dimensional self-adjoint extensions of this operator, which can be parameterized by  $\gamma\in [0,\pi)$. To describe them we introduce a real valued function\footnote{There are also complex valued functions see Sect.\ref{SF10a}. But since we have in mind application to the water wave theory it is reasonable to consider here only real value functions $h$}
$$
w_\gamma =\sin (\kappa\log \frac{1}{2}r+\gamma)\cosh(\kappa\theta),
$$
where $(r,\theta)$ are polar coordinates near $(x,y)=(0,\eta(0)$:
\begin{equation}\label{D15a}
x=r\sin\theta,\;\;\eta(0)-y=r\cos\theta
\end{equation}
and $\kappa$ is the positive root of the equation
\begin{equation}\label{K49zq}
\kappa\tanh\Big(\kappa\alpha^*\Big)=\rho_0.
\end{equation}
 The function $ \widehat{w}$ is chosen to satisfy
\begin{equation}\label{F22ap}
{\mathcal L}\widehat{w}\in L^2(\Omega)\;\;\mbox{and $\widehat{w}$ satisfies all homogeneous boundary conditions in (\ref{K2bbzz})}
\end{equation}
and
\begin{equation}\label{F22a}
\widehat{w}=\zeta(r)w_\gamma+w,\;\;\mbox{$w\in V^{2}_{\beta_*}(\Omega)$ with some $\beta_*\in (1-\alpha,1)$},
\end{equation}
where $\zeta$ is a smooth cut-off function equal $1$ for $r<\delta$ and $0$ for $r>2\delta$ ($\delta$ is a small positive number). The existence of such function (with any $\beta_*\in (1-\alpha,1)$ is proved in Proposition \ref{P22aa}(i) and, moreover, it is shown there if we have two  functions $\widehat{w}_1$ and $\widehat{w}_2$ satisfying (\ref{F22ap}) and (\ref{F22ap}) then $\widehat{w}_1-\widehat{w}_2\in V^{2}_0(\Omega)$  and the choice of function $\widehat{w}$ does not depend on the choice of $\zeta$.

We define ${\mathcal D}_\gamma$ as the space of functions consisting of the sums
\begin{equation}\label{K49z}
u=C\widehat{w}+v\,:\;
  \,v\in V^{2}_0(\Omega)\;\;\mbox{$v$ satisfies (\ref{K2bbzz}) and $C$ is a constant}.
\end{equation}

We denote the operator ${\mathcal L}$ with the domain ${\mathcal D}_\gamma$ by ${\mathcal L}_\gamma$. The main theorem of this paper is the following

\begin{theorem}\label{Th1.1} For any $\gamma\in [0,\pi)$ the operator ${\mathcal L}_\gamma$ is self-adjoint has a discrete spectrum consisting of eigenvalues of finite multiplicity. Moreover this operator has infinitely many negative eigenvalues. Large negative eigenvalues are exhausted by $\lambda_k=-s_k^2$, where
\begin{equation}\label{K49s}
s_k=e^{(\gamma+\gamma_\kappa+k\pi)/\kappa}\big(1+O(e^{-\alpha k/\kappa}))\big)\;\;\mbox{as $k\to\infty$},
\end{equation}
where $k$ is a large integer and $\gamma_\kappa$ is a real constant defined by
\begin{equation}\label{J25a}
\Gamma(1+i\kappa)=\Big(\frac{\pi\kappa}{\sinh(\pi\kappa)}\Big)^{1/2}e^{i\gamma_\kappa}.
\end{equation}
Moreover, the eigenvalues {\rm (\ref{K49s})} are simple.
\end{theorem}

Let us explain the main ideas of the proof. The operator ${\mathcal L}$ corresponding to the boundary value problem (\ref{K2bb}), (\ref{K2bbzz})  is symmetric in the subspace $\widetilde{V}^2_0(\Omega)$ of $V^2_0(\Omega)$ defined by the boundary conditions (\ref{K2bbzz}). The first step is to find self-adjoint extensions of this operator. Similar problems are discussed in papers \cite{Krall1982,Marlettta2009,Nazarov2018} for one dimensional Sturm-Liouville problem, two-dimensional problem with Robin boundary condition in a disc and in a domain with smooth boundary respectively. Here in Sect.\ref{Sselfadj}, we obtain self-adjoint extensions of the operator by using an asymptotic approach similar to that in \cite{Nazarov2018}. It appears that this asymptotic approach can be used for description of self-adjoint extensions of the model problems in the angle, on the half-line and on an interval and all these extensions are naturally obtained from each other (see Sect.\ref{Shalfl}, \ref{SecM7}). The second step is an one-dimensional spectral problems on a half-line and on an interval. The spectral problem on a half-line is presented in Sect.\ref{Shalfl} and all results are borrowed from \cite{Krall1982}. The spectral problem on an interval, considered in Sect.\ref{SecM7}, is an important step in the proof of the main theorem since it gives the leading term in asymptotics of negative eigenvalues and corresponding eigenfunctions. Two boundary conditions are needed there. The condition at zero comes from the self-adjoint extension of the operator. The boundary condition at another end of the interval is taken as a Robin condition and it will be justified later in Sect.\ref{SM10a}. We assume there that it is already found and we take it in the required form from the beginning. Next step is devoted to a $2$D model problem in a domain close to $\Omega$ but the coefficients and the free surface are replaced by the main terms in their asymptotics near the corner, see Sect.\ref{SM10a}, \ref{SM7aa} and \ref{SKKa}. In Sect.\ref{SM10a} we derive the second boundary condition for the spectral problem on the interval. It comes from a one-dimensional Dirichlet-Neumann mapping obtaining as a result of solving $2$D problem depending on a parameter. In Sect.\ref{SKKa} we obtain weighted estimate for solutions to the $2$D model problem, where the spectral parameter is considered as a parameter. In Sect.\ref{SM7aa} we obtain asymptotics for the eigenvalues and for the eigenfunctions of the $2$D model spectral problem. The last step is consideration of the general $2$D spectral problem as a perturbation of the model $2$D problem, see Sect.\ref{Sproof}. Since the distance between neighbour eigenvalues is comparable with the absolute values of the corresponding eigenvalues, we can applied the technique developed for perturbation of isolated eigenvalue, see Sect.\ref{Spert}. Difficulties here come from the fact that the domains of self-adjoint operators is not a Sobolev space but its extension by a certain function. This part requires a careful analytic considerations. In remark \ref{RM1} we give an asymptotic formula for the eigenfunctions corresponding to large negative eigenvalues.

\section{Model problems}\label{Sec11}

Here we present some auxiliary problems, which will play an important role in the proof of Theorem \ref{Th1.1}.

\subsection{Model problem in an angle}\label{SF10a}

Let $A$ be the angle
$$
A=\{ (x,y)\,:\,r>0,\;\;\theta\in (0,\alpha^*)\}.
$$
 Consider the equation
\begin{equation}\label{K14}
\Delta u=f\;\;\mbox{in $A$}
\end{equation}
with boundary conditions
\begin{eqnarray}\label{K1}
&&r^{-1}(\partial_\theta u-\rho_0)u=g\;\;\mbox{for $\theta=\alpha^*$}\nonumber\\
&&\partial_\theta u=0\;\;\mbox{for $\theta=0$}.
\end{eqnarray}

{\bf Homogeneous problem.} First let us construct all solutions to the homogeneous problem (\ref{K14}), (\ref{K1}), i.e. with $f=0$ and $g=0$. It can be done by separation of variables. There are two solutions of the form
$$
u(t,\theta)=r^{\pm i\kappa}\cosh (\kappa\theta),
$$
where $\kappa$ is a real positive number  satisfying (\ref{K49zq}).
 We  denote by ${\mathcal X}$ the $2D$-space of functions
\begin{equation}\label{K4s1z}
w=w(r,\theta)=ar^{i\kappa }\cosh (\kappa\theta)+br^{-i\kappa}\cosh (\kappa\theta),\;\;a,b\in\Bbb C.
\end{equation}
Let us introduce the following symplectic form on ${\mathcal X}$
\begin{equation}\label{K4s}
q(w_1,w_2)
=\int_0^{\alpha^*}(\partial_rw_1(r,\theta)\overline{w_2(r,\theta)}-w_1(r,\theta)\overline{\partial_rw_2(r,\theta)})rd\theta.
\end{equation}
Using Green's formula one can verify that the expression in the right-hand side is independent of $r$.
This form is non-generate on ${\mathcal X}$ and represents the Wronskian of two solutions to a corresponding ODE in  the $r$ variable.

Let
$$
{\mathcal Y}_0=\{w\in {\mathcal X}\,:\, q(w,w)=0\}.
$$
Direct calculation shows that
\begin{equation*}
{\mathcal Y}_0=\{w\in {\mathcal X}\,:\, |a|=|b|\;\; \mbox{in (\ref{K4s1z})} \}.
\end{equation*}
If we denote by ${\mathcal Y}_{r}$ real valued functions from ${\mathcal Y}_0$, then
\begin{equation}\label{K4s1}
{\mathcal Y}_{r}=\{w\in {\mathcal X}\,:\,w=a\sin(\kappa \log \frac{1}{2}r+\gamma)\cosh (\kappa\theta),\;\;a\in\Bbb R,\;\gamma\in [0,\pi)\}.
\end{equation}

To see that the form $q$ is non-degenerating we put
$$
U_+(r,\theta)=\sin(\kappa \log \frac{1}{2}r+\gamma_1)\cosh (\kappa\theta),\;\;U_-(r,\theta)=\sin(\kappa \log \frac{1}{2}r+\gamma_2)\cosh (\kappa\theta).
$$
 Then
$$
q(U_+,U_-)
=\kappa\sin(\gamma_2-\gamma_1)\int_0^{\alpha^*}\cosh^2(\kappa\theta)d\theta.
$$

The remaining solutions to the homogeneous problem (\ref{K14}), (\ref{K1}) have the form
$$
u(r,\theta)=r^{\mu }\cos (\mu\theta),
$$
where $\mu$ satisfies (\ref{K5}).
We numerate the positive roots of (\ref{K5}) according to
$\mu_k\alpha^*\in ((k-1)\pi+\frac{\pi}{2},k\pi)$, $k=1,\ldots$. Clearly, $-\mu_k$  also solves (\ref{K5}). If we denote
$$
v_0(\theta)=\cosh(\kappa\theta),\;\;v_k(\theta)=\cos(\mu_k\theta),\;\;k=1,2,\ldots,
$$
then the system $\{\varphi_k\}_{k=0}^\infty$ is an orthogonal basis in $L^2(0,\alpha^*)$. Let
$$
\hat{v}_k^2=\int_0^{\alpha^*}\cos^2(\mu_k\theta)d\theta=\frac{\alpha^*}{2}+\frac{\sin(2\mu_k\alpha^*)}{4\mu_k},
\;\;k=1,\ldots
$$
and
$$
\hat{v}_0^2=\int_0^{\alpha^*}\cosh^2(\kappa\theta)d\theta=\frac{\alpha^*}{2}+\frac{\sinh(2\kappa\alpha^*)}{4\kappa}.
$$
Now the system
\begin{equation}\label{F7a}
\phi_k(\theta)=\frac{1}{\hat{v}_k}v_k(\theta),\;\;k=0,1,\ldots,
\end{equation}
is an orthonormal basis in $L^2(0,\alpha^*)$.

\bigskip
{\bf Non-homogeneous problem (\ref{K14}), (\ref{K1})}.
For  $\beta\in\Bbb R$ and integer $l\geq 0$,  we introduce the following spaces: $V^{l}_\beta(A)$ consists of functions in $A$  with the norm
$$
||u||_{V^{l}_\beta(A)}=\Big(\int_A\sum_{i+j\leq l}|\partial_x^i\partial_y^ju|^2r^{2(\beta-(l-i-j))}dxdy\Big)^{1/2},
$$
 the space $L^2_\beta(A)$ coincides with $V^0_\beta(A)$.
The space $V^{1/2}_\beta(0,\infty)$ consists of functions defined on the ray $\theta=\alpha^*$ and has the norm
$$
||g||_{V^{1/2}_\beta(0,\infty)}=
\Big(\int_0^\infty\int_0^\infty\frac{|r^\beta g(r)-s^\beta g(s)|^2}{|r-s|^2}drds+\int_0^\infty r^{2\beta-1}|g|^2dr\Big)^{1/2}.
$$
Another equivalent norm is the following (see \cite{MP78})
$$
\inf\Big\{||u||_{V^{1}_\beta(A)}\,:\,u\in V^{1}_\beta(A),\, u|_{\theta=\alpha^*}=g\Big\}.
$$
We will omit the index $\beta$ in the notation of spaces if $\beta=0$.

The main solvability result for (\ref{K14}), (\ref{K1}) is the following

\begin{proposition}\label{PF22} {\rm (i)} Let $\beta-1\neq 0$ and $\beta+\mu_j-1\neq 0$ for $j=\pm 1,\pm 2,\ldots$. Then for each $f\in L^2_\beta(A)$ and $g\in V^{1/2}_\beta(0,\infty)$ there exists a unique $u\in V^{2}_\beta(A)$ solving {\rm (\ref{K14})}, {\rm (\ref{K1})} and the following estimate holds
$$
||u||_{V^{2}_\beta(A)}\leq C(||f||_{L^2_\beta(A)}+||g||_{V^{1/2}_\beta(0,\infty)}).
$$

{\rm (ii)} Let $\beta_j$, $j=1,2$, satisfy $1-\mu_1<\beta_1<1<\beta_2<1+\mu_1 $
  and let $f\in L^{2}_{\beta_1}(A)\bigcap L^{2}_{\beta_2}(A)$ and $g\in V^{1/2}_{\beta_1}(0,\infty)\bigcap V^{1/2}_{\beta_2}(0,\infty)$. Then
$$
u_2-u_1=(c_+r^{i\kappa }+c_-r^{^-i\kappa })\phi_0(\theta),
$$
where $u_j\in V^{2}_{\beta_j}(A)$  the solutions from {\rm (i)} for $j=1,2$ and $c_{\pm}\in\Bbb C$.

{\rm (iii)} Let $\beta_j$, $j=1,2$, satisfy $1-\mu_1<\beta_1,\beta_2<1 $
  and let $f\in L^{2}_{\beta_1}(A)\bigcap L^{2}_{\beta_2}(A)$ and $g\in V^{1/2}_{\beta_1}(0,\infty)\bigcap V^{1/2}_{\beta_2}(0,\infty)$. Then $u_2=u_1$,
where $u_j\in V^{2}_{\beta_j}(A)$   solutions from {\rm (i)} for $j=1,2$.
\end{proposition}

The theory of boundary value problem for elliptic equations in an angle is well developed and for the proof of such assertions we refer to books \cite{Nazarov1994} and \cite{KMR97} and references there.

\subsection{A model spectral problem on a half-line}\label{Shalfl}

Let $\lambda=-\tau^2$, where $\tau$ is a positive number. Consider the spectral problem
\begin{equation}\label{Okt26a1x}
{\mathcal M}h=-\frac{d}{rdr}\Big(r\frac{d}{dr}h(r)\Big)-\frac{\kappa^2}{r^2}h(r)=-\tau^2 h(r)\;\;\mbox{for $r\in(0,\infty)$}.
\end{equation}
Let  $\widehat{W}^2_\beta(0,\infty)$  be the space of functions $v$ on $(0,\infty)$ with finite norm
$$
||v||_{\widehat{W}^2_\beta(0,\infty)}=\Big(\int_0^\infty r^{2\beta}\Big(|v^{''}|^2+(1+r^{-2})|v'|^2+(1+r^{-4})|v|^2\Big)rdr\Big)^{1/2}.
$$
It can be described equivalently as $v\in \widehat{V}^2_\beta(0,\infty)\bigcap \widehat{L}^2_\beta(0,\infty)$, where $\widehat{V}^2_\beta(0,\infty)$ and $\widehat{L}^2_\beta(0,\infty)$ are the spaces of functions with the finite norms
$$
||v||_{\widehat{V}^2_\beta(0,\infty)}=\Big(\int_0^\infty r^{2\beta}(|v^{''}|^2+r^{-2}|v'|^2+r^{-4}|v|^2)rdr
$$
and
$$
||v||_{\widehat{L}^2_\beta(0,\infty)}=\Big(\int_0^\infty r^{2\beta}|v|^2rdr\Big)^{1/2}
$$
respectively. Then the operator ${\mathcal M}$ is symmetric on $X_0$ and its self-adjoint extension is defined on a domain
\begin{equation}\label{J26a}
D_{\mathcal M}=\{ h=C\zeta(\tau r)\sin(\kappa\log \frac{1}{2}r+\gamma)+v\;:\;v\in \widehat{W}^2_0(0,\infty) \;\;\mbox{and $C$ is constant}\}.
\end{equation}
Here $\gamma\in [0,\pi)$ is a fixed constant and $\zeta$ is a smooth cut-off function equal $1$ for small $r$ and zero for large $r$. For the fact that the operator ${\mathcal M}$ considered in the space $\widehat{L}^2(0,\infty)$  with the domain $D_{\mathcal M}$ is self-adjoint we refer to \cite{Krall1982} (see also Sect.\ref{Sselfadj} in this paper, where a more general situation is discussed). Clearly the definition of the domain $D_{\mathcal M}$ as well as the constant $C$ in its definition does not depend on the choice of the cut-off function $\zeta$ (but the function $v$ may depend on the choice of $\zeta$).  We supply the space $D_{\mathcal M}$ with the norm
\begin{equation}\label{J26aa}
||h||_{D_{\mathcal M}}=\Big(|C|^2+||v||^2_{\widehat{W}^2_0(0,\infty)}\Big)^{1/2},
\end{equation}
where $C$ and $v$ are the same as in the definition (\ref{J26a}).

According to \cite{Dunster1990} linear independent solutions to (\ref{Okt26a1x}) are $K_{i\kappa}(\tau r)$ and  $I_{i\kappa}(\tau r)$, where $K_{i\kappa}=K_{i\kappa}(z)$ and $I_{i\kappa}=I_{i\kappa}(z)$ are Bessel's functions of imaginary order. They have the following asymptotics (see \cite{Dunster1990})
\begin{equation}\label{F28aa}
K_{i\kappa}(z)=\Big(\frac{\pi}{2z}\Big)^{1/2}e^{-z}\Big(1+O\Big(\frac{1}{z}\Big)\Big),\;\;
I_{i\kappa}(z)=\Big(2\pi z\Big)^{1/2}e^{z}\Big(1+O\Big(\frac{1}{z}\Big)\Big)
\end{equation}
for $z\to\infty$ and
\begin{equation}\label{F23a}
K_{i\kappa}(z)=-\Big(\frac{\pi}{\kappa\sinh(\pi\kappa)}\Big)^{1/2}
\sin\big(\kappa\ln\big(\frac{1}{2}z\big)-\gamma_\kappa\big)+O(z^2),
\end{equation}
\begin{equation}\label{F23aa}
I_{i\kappa}(z)=\Big(\frac{\sinh(\pi\kappa)}{\pi\kappa}\Big)^{1/2}
\cos\big(\kappa\ln\big(\frac{1}{2}z\big)-\gamma_\kappa\big)+O(z^2)
\end{equation}
as $z\to 0$. Here $\gamma_\kappa$ is a real constant defined by (\ref{J25a}).

The following theorem is proved in \cite{Krall1982}

\begin{theorem}\label{TM10} Continuous spectrum of ${\mathcal M}$ coincides with the positive half-line $[0,\infty)$ and the negative axis contains only isolated simple eigenvalues $-\tau_k^2$, where
\begin{equation}\label{F1a}
\tau_k=2e^{(\gamma_\kappa+\gamma)/\kappa}e^{k\pi/\kappa},\;\;k=0,\pm 1,\pm 2,\ldots,
\end{equation}
with corresponding eigenfunctions $K_{i\kappa}(\tau_k r)$.
\end{theorem}

We note that $\{\tau_k\}$ represents a geometric sequence with  the common ratio
\begin{equation}\label{F10bb}
q=e^{\pi/\kappa}.
\end{equation}

We continue this section with solvability results for the nonhomogeneous equation
\begin{equation}\label{J26cb}
{\mathcal M}h+\tau^2h=f.
\end{equation}

\begin{lemma}\label{LK1} Let   $\tau\in (\tau_k/q,q\tau_k)$, $\tau\neq\tau_k$, for certain $k\in\Bbb Z$. Let also
\begin{equation}\label{J26ca}
h=C\zeta(\tau r)\sin(\kappa\log \frac{1}{2}r+\gamma)+v\;:\;v\in X_0
\end{equation}
satisfy (\ref{J26cb}) with $f\in \widehat{L}^2(0,\infty)$.
Then
\begin{equation}\label{J26ccc}
C=\Big(\frac{\kappa\sinh(\pi\kappa)}{\pi}\Big)^{1/2}
\frac{1}{\kappa\sin(\kappa\log(\tau_k/\tau))}\int_0^\infty K_{i\kappa}(\tau r)f(r)rdr.
\end{equation}

\end{lemma}
\begin{proof}  From (\ref{J26cb}) it follows
$$
\int_s^\infty fK_{i\kappa}(\tau r)rdr=\int_s^\infty({\mathcal M}h+\tau^2h)K_{i\kappa}(\tau r)rdr=
(\partial_rhK_{i\kappa}(\tau r)-h\partial_rK_{i\kappa}(\tau r))|_{r=s}.
$$
Since
\begin{eqnarray*}
&&r(\partial_rhK_{i\kappa}(\tau r)-h\partial_rK_{i\kappa}(\tau r))|_{r=s}\to
\kappa\Big(\frac{\pi}{\kappa\sinh(\pi\kappa)}\Big)^{1/2}
\Big(\cos\big(\kappa\log\big(\frac{1}{2}\tau r\big)-\gamma_\kappa\big)\sin(\kappa\log r+\gamma)\\
&&-
\sin\big(\kappa\log\big(\frac{1}{2}\tau r\big)-\gamma_\kappa\big)\cos(\kappa\log r+\gamma)\Big)
=\kappa\Big(\frac{\pi}{\kappa\sinh(\pi\kappa)}\Big)^{1/2}\sin(\gamma+\gamma_\kappa-\kappa\log\big(\frac{1}{2}\tau))
\end{eqnarray*}
as $s\to 0$ we arrive at
\begin{equation}\label{J26cc}
\kappa\Big(\frac{\pi}{\kappa\sinh(\pi\kappa)}\Big)^{1/2}\sin(\gamma+\gamma_\kappa-\kappa\log\big(\frac{1}{2}\tau\big)\big)C
=\int_0^\infty K_{i\kappa}(\tau r)f(r)rdr.
\end{equation}
By (\ref{F1a})
$$
\gamma+\gamma_\kappa-\kappa\log\big(\frac{1}{2}\tau_k\big)=-k\pi .
$$
Hence,
$$
\sin(\gamma+\gamma_\kappa-\kappa\log\big(\frac{1}{2}\tau)\big)
=\sin(\kappa\log\big(\frac{1}{2}\tau_k)\big)-\kappa\log\big(\frac{1}{2}\tau\big)\big)=\sin(\kappa \tau_k/\tau).
$$
Therefore formula (\ref{J26cc}) can be written as (\ref{J26ccc}).

\end{proof}

In order to include in our considerations the case $\tau=\tau_k$, we introduce  the following functions
$$
n(\tau/\tau_k)=\frac{\sin(\kappa\log(\tau_k/\tau))}{\tau/\tau_k-1},\;\;\;m(\tau/\tau_k,r)=\frac{K_{i\kappa}(\tau r/\tau_k )-K_{i\kappa}(r)}{\tau/\tau_k-1}.
$$
We note that $-\pi<\kappa\log(\tau_k/\tau))<\pi$ is equivalent to $\tau\in (\tau_k/q,q\tau_k)$. The last inclusion guarantees that $n(\tau/\tau_k)\neq 0$.

\begin{lemma}  Additionally to assumptions of {\rm Lemma \ref{LK1}}, we assume that
\begin{equation}\label{F11ba}
\int_0^\infty K_{i\kappa}(\tau_k r)f(r)rdr=0.
\end{equation}
Then the following representation for the constant $C$ in the representation (\ref{J26ca})  holds
\begin{equation}\label{J26ccx}
\kappa C n(\tau,\tau_k)=\Big(\frac{\kappa\sinh(\pi\kappa)}{\pi}\Big)^{1/2}
\int_0^\infty f(r)\,m(\tau/\tau_k,\tau_kr)rdr.
\end{equation}
\end{lemma}
\begin{proof} The proof follows immediately from (\ref{J26ccc}) and (\ref{F11ba}).

\end{proof}

\begin{remark} Let $q_0\in (0,q)$ and let $\tau\in [\tau_k/q_0,q_0\tau_k]$.  Let also $\beta<1$ and $f\in\widehat{L}^2_\beta(0,\infty)$.
The representation (\ref{J26ccc}) implies
\begin{equation}\label{F1b}
|C|\leq c\tau^{\beta-1}\frac{1}{|\sin\kappa\log(\tau_k/\tau)|}\Big(\int_0^\infty r^{2\beta}|f|^2rdr\Big)^{1/2},
\end{equation}
where $c$ is independent of $\tau$ and $\tau_k$. Using that
$$
|\sin\kappa\log(\tau_k/\tau)|\geq c\frac{|\tau-\tau_k|}{\tau}
$$
we get
\begin{equation}\label{F1ba}
|C|\leq c\frac{\tau^{\beta}}{|\tau-\tau_k|}\Big(\int_0^\infty r^{2\beta}|f|^2rdr\Big)^{1/2}.
\end{equation}
If we assume additionally that (\ref{F11ba}) is valid then
\begin{equation}\label{F1baz}
|C|\leq c\tau^{\beta-1}\Big(\int_0^\infty r^{2\beta}|f|^2rdr\Big)^{1/2}.
\end{equation}
\end{remark}

\begin{lemma} Let $q_0\in (0,q)$ and let $\tau\in [\tau_k/q_0,q_0\tau_k]$, $\tau\neq\tau_k$, where $k\in\Bbb Z$. Let $h\in {\mathcal D}_{\mathcal M}$ satisfy {\rm (\ref{J26cb})}.
Then
\begin{equation}\label{J26b}
|C|^2+\int_0^\infty (|\partial_rv|^2+\tau^2|v|^2)rdr\leq \frac{c}{|\tau-\tau_{k}|^2}\int_0^\infty |f|^2rdr
\end{equation}
and
\begin{equation}\label{J26baq}
||v||^2_{\hat{V}^2_0(0,\infty)}\leq \frac{c\tau^2}{|\tau-\tau_{k}|^2}\int_0^\infty |f|^2rdr,
\end{equation}
where $c$ does not depend on $\tau$ and $f$, but depends on $q_0$. Here $C$ and $v$ are the same as in {\rm (\ref{J26ca})}.

Let additionally {\rm (\ref{F11ba})} be valid and
\begin{equation}\label{M4a}
\int_0^\infty h(r)K_{i\kappa}(\tau_k r)rdr=0.
\end{equation}
 Then
\begin{equation}\label{J26bz}
|C|^2+\tau^{-2}||v||^2_{\hat{V}^2_0(0,\infty)}+\int_0^\infty (|\partial_rv|^2+\tau^2|v|^2)rdr\leq \frac{c}{\tau^2}\int_0^\infty |f|^2rdr
\end{equation}
for all $\tau\in [\tau_k/q_0,q_0\tau_k]$.
\end{lemma}
\begin{proof} The estimate (\ref{J26b}) for the constant $C$ in (\ref{J26ca}) follows from (\ref{F1ba}) with $\beta=0$.
Since the operator is self-adjoint we get
$$
\int_0^\infty |h|^2rdr\leq \frac{c}{|\tau^2-\tau_k^2|^2}\int_0^\infty |f|^2rdr.
$$
Using the representation (\ref{J26ca}) and the above estimate together with the estimate for $C$, we obtain the estimate
\begin{equation}\label{F17a}
\tau^2\int_0^\infty |v|^2rdr\leq \frac{c}{|\tau-\tau_k|^2}\int_0^\infty |f|^2rdr.
\end{equation}

Furthermore, the function $v\in X_0$ solves the problem
$$
{\mathcal M}v+\tau^2v=F=f-C[{\mathcal M},\zeta]\sin(\kappa\log\frac{1}{2}r+\gamma)-\tau^2\zeta\sin(\kappa\log\frac{1}{2}r+\gamma)\in \widehat{L}^2(0,\infty).
$$
This implies the estimate
\begin{equation}\label{F1b}
||v||_{\widehat{V}^2_0(0,\infty)}\leq c||F-\tau^2v||_{\widehat{L}^2(0,\infty)}
\end{equation}
which leads to (\ref{J26baq}).

The estimate for $\partial_rv$ can be obtained from (\ref{F17a}) and (\ref{J26b}).

The proof of (\ref{J26bz}) is basically the same but instead of (\ref{F1ba}) we must use (\ref{F1baz}) with $\beta=0$.

\end{proof}

Let us estimate a weighted norm of $v$ in (\ref{J26ca}). We introduce  the space
\begin{equation}\label{J26ca1}
D_{\mathcal{M}}^\beta=\{h=C\zeta(\tau r)\sin(\kappa\log \frac{1}{2}r+\gamma)+v\;:\;v\in \widehat{W}^2_\beta(0,\infty)\;\;\mbox{and $C$ is constant}\}
\end{equation}
with the norm
$$
||h||_{D_{\mathcal{M}}^\beta}=|C|+||v||_{\widehat{W}^2_\beta(0,\infty)}.
$$
If $\beta\leq 1$ then the first term in the right-hand side in (\ref{J26ca1}) does not belong to $\widehat{W}^2_\beta(0,\infty)$.


\begin{lemma}\label{LF1} Let $q_0\in (0,q)$, $\tau\in [\tau_k/q_0,q_0\tau_k]$ and let
  $\beta \in (-1,1)$.  If $h\in D_{\mathcal{M}}^\beta$ satisfies {\rm (\ref{J26cb})} with $f\in \widehat{L}^2_\beta(0,\infty)$
 then
\begin{equation}\label{J26bq}
\tau^{-2\beta}|C|^2+\int_0^\infty r^{2\beta}(|\partial_rv|^2+\tau^2|v|^2)rdr\leq \frac{c}{|\tau-\tau_{k}|^2}\int_0^\infty r^{2\beta}|f|^2rdr
\end{equation}
and
\begin{equation}\label{J26ba}
||v||^2_{\widehat{V}^2_\beta(0,\infty)}\leq c\frac{c\tau^{2}}{|\tau-\tau_{k}|^2}\int_0^\infty r^{2\beta}|f|^2rdr
\end{equation}
where $c$ does not depend on $\tau$ and $f$.

If additionally, {\rm (\ref{F11ba})} and {\rm (\ref{M4a})} be valid. Then
\begin{equation}\label{J26bzz}
\tau^{-2\beta}|C|^2+\tau^{-2}||v||^2_{\widehat{V}^2_\beta(0,\infty)}+\int_0^\infty r^{2\beta}(|\partial_rv|^2+\tau^2|v|^2)rdr\leq c\tau^{-2}\int_0^\infty r^{2\beta}|f|^2rdr.
\end{equation}
\end{lemma}

\begin{proof} The estimates (\ref{J26bq}) and (\ref{J26bzz}) for the constant $C$ follow from (\ref{F1ba}) and (\ref{F1baz}) respectively.

Next the equation for $v$ can be written as
 \begin{equation}\label{M2a}
{\mathcal M}v+\tau^2v=F:=f-C[{\mathcal M},\zeta(\tau r)\sin(\kappa\log \frac{1}{2}r+\gamma)]-C\tau^2\zeta(\tau r)\sin(\kappa\log \frac{1}{2}r+\gamma),
\end{equation}
where $F\in \widehat{L}^2_\beta(0,\infty)$. Moreover
$$
||F||_{\widehat{L}^2_\beta(0,\infty)}\leq c(||f||_{\widehat{L}^2_\beta(0,\infty)}+\tau^{1-\beta} |C|).
$$

Using the change of variable $R=\tau r$ we transform the problem (\ref{M2a}) to
\begin{equation}\label{F3ba}
{\mathcal M}v+v=F_\tau :=\tau^{-2}F.
\end{equation}

 We represent the right-hand side in (\ref{F3ba}) as $F=F_1+F_2+F_3$, where $F_1(R)=F_\tau(R)$ for $R<\delta$ and zero otherwise, $f_3(R)=F_\tau(R)$ for $R>N$ and zero otherwise, where $\delta$ and $N$ are small and large positive numbers respectively. Then let $v_1,v_3\in \widehat{V}^2_\beta(0,\infty)$ be  solutions to
$$
{\mathcal M}v_1+v_1=F_1\;\;\mbox{on $(0,3\delta)$},\;\;{\mathcal M}v_3+v_3=F_3\;\;\mbox{on $(N/2,\infty)$.}
$$
We can choose them to satisfy
$$
||v_1||_{\widehat{V}^2_\beta(0,3\delta)}\leq c||F_1||_{\widehat{L}^2_\beta},\;\;||v_3||_{\widehat{W}^2_\beta(N/2,\infty)}\leq c||F_3||_{\widehat{L}^2_\beta}
$$
Let $\zeta_1$ and $\zeta_3$ be two $C^2$ cut-off functions such that $\zeta(r)=1$ for $r<2\delta$ and $\zeta(r)=0$ for $r>3\delta$ and $\zeta_3(r)=1$ for $r>3N/4$ and $\zeta_3(r)=0$ for $r<N/2$. Then the function $v_2=v-\zeta_1v_1-\zeta_3v_3$ satisfies the equation
\begin{equation}\label{F3a}
{\mathcal M}v_2+v_2=F_2-[{\mathcal M},\zeta_1]v_1-[{\mathcal M},\zeta_3]v_3=:{\mathcal F}
\end{equation}
One can verify that the support of ${\mathcal F}$ belongs to $[\delta,N/2]$ and
$$
||{\mathcal F}||_{\widehat{L}^2_ \beta(0,\infty)}\leq c||F||_{\widehat{L}^2_\beta(0,\infty)}.
$$
So the equation (\ref{F3a}) has solution in $\widehat{L}^2(0,\infty)$ and it satisfies
$$
||v_2||_{V^2_0(0,\infty)}+||v_2||_{L^2(0,\infty)}\leq c\frac{1}{\sin\delta_*}||F||_{\widehat{L}^2_\beta(0,\infty)}.
$$
Using local estimates near $0$ and $\infty$ and that ${\mathcal F}$ vanishes there, we conclude that
$$
||v_2||_{\widehat{V}^2_\beta(0,\infty)}+||v_2||_{\widehat{L}^2_\beta(0,\infty)}\leq c\frac{1}{\sin\delta_*}||f||_{\widehat{L0}^2_\beta(0,\infty)}.
$$
which proves Lemma.
\end{proof}

\subsection{Bounded interval}\label{SecM7}

Consider the following spectral problem on an interval of length $\delta$:
\begin{equation}\label{Okt26a1}
{\mathcal M}h=-\tau^2 h(r)\;\;\mbox{for $r\in(0,\delta)$}
\end{equation}
with boundary condition
\begin{equation}\label{Okt26av}
h'(\delta)+(\tau-\alpha(\tau^{-1}))h(\delta)=0
\end{equation}
where $\alpha(s)$ is a $C^\infty$ function in a neighborhood of the origin\footnote{The parameter $\tau$ here is included in the boundary condition also. So this is actually a boundary value problem with a parameter $\tau\geq\tau_0$, where $\tau_0$ is sufficiently large. The definition of eigenvalues and eigenfunctions of such problems is standard.}. We will always assume in such problem that $\tau\delta$ is sufficiently large.
Let also $\widehat{V}^2_\beta(0,\delta)$  be the space of functions $v$ on $(0,\delta)$ with finite norm
$$
||v||_{\widehat{V}^2_\beta(0,\delta)}=\Big(\int_0^\delta r^{2\beta}\Big(|v^{''}|^2+r^{-2}|v'|^2+r^{-4}|v|^2\Big)rdr\Big)^{1/2}.
$$
The operator ${\mathcal M}$ is symmetric on the subspace  of $\widehat{V}^2_\beta(0,\delta)$ defined by (\ref{Okt26av}). We will consider the operator ${\mathcal M}$ on the domain
$$
\widehat{D}_\gamma(\tau)=\{h=C\zeta(\tau r) \sin(\kappa \log\frac{1}{2} r+\gamma)+v\,:\,v\in \widehat{V}^2_\beta(0,\delta),\;\mbox{$v$ satisfies (\ref{Okt26av}) and $C$ is a constant}.
$$
Here $\gamma\in [0,\pi)$ is a fixed constant and $\zeta$ is a cut-off function equal $1$ for $r<1/3$ and $0$ for $r>2/3$.
Since   the operator
\begin{equation}\label{F9a}
{\mathcal M}(\tau):\widehat{D}_\gamma(\tau)\to \hat{L}^2(0,\delta)
\end{equation}
is self-adjoint for each $\tau\geq\tau_0$, where $\tau_0$ is sufficiently large, it is also Fredholm with zero index.

To find values of $\tau$ for which the kernel of the operator (\ref{F9a}) is non-trivial we are looking for solution in the form
$$
h(r)= K_{i\kappa}(\tau r)-Q(\tau) I_{i\kappa}(\tau r),\;\;\mbox{$Q(\tau)$ is  a function of $\tau$},
$$
which is subject to
\begin{equation}\label{F28a}
h'(\delta)+(\tau-\alpha(\tau^{-1}))h(\delta)=0,\;\;
h(r)=C\sin\big(\kappa\log\frac{1}{2}r+\gamma\big)+O(r)\;\;\mbox{as $r\to 0$},
\end{equation}
where $C$ is a constant. Using the first equation in (\ref{F28a}) and asymptotic expansions (\ref{F28aa}), we can find $Q$:
\begin{equation}\label{F15a}
Q(\tau)=\tau^{-2} e^{-2\tau \delta}m(\tau^{-1}),
\end{equation}
where $m(s)$ is $C^\infty$ in a neighborhood of the origin.
Furthermore, the second relation in (\ref{F28a}) together with the asymptotic formulas (\ref{F23a}) and (\ref{F23aa}) implies
\begin{eqnarray*}
&&-\Big(\frac{\pi}{\kappa\sinh(\pi\kappa)}\Big)^{1/2}
\sin\big(\kappa\log\big(\frac{1}{2}\tau r\big)-\gamma_\kappa\big)-Q\Big(\frac{\sinh(\pi\kappa)}{\pi\kappa}\Big)^{1/2}
\cos\big(\kappa\log\big(\frac{1}{2}\tau r\big)-\gamma_\kappa\big)\\
&&= C\sin\big(\kappa\log\big(\frac{1}{2}r\big)+\gamma\big)
\end{eqnarray*}
Thus
\begin{eqnarray}\label{Okt18a}
&&-\Big(\frac{\pi}{\kappa\sinh(\pi\kappa)}\Big)^{1/2}\Big(\sin\big(\kappa\log\big(\frac{1}{2}\tau r\big)-\gamma_\kappa\big)+A
\cos\big(\kappa\log\big(\frac{1}{2}\tau r\big)-\gamma_\kappa\big)\Big)\nonumber\\
&&= C\sin\big(\kappa\log\big(\frac{1}{2}r\big)+\gamma\big),
\end{eqnarray}
where
\begin{equation}\label{F15aa}
A(\tau)=\frac{Q(\tau)\sinh(\pi\kappa)}{ \pi}.
\end{equation}
Now define the angle $\psi$ by the relations
$$
\cos\psi=\frac{1}{\sqrt{1+A^2}},\;\;\sin\psi=\frac{A}{\sqrt{1+A^2}}
$$
or
\begin{equation}\label{F15ab}
\psi(\tau)=\arctan A(\tau).
\end{equation}
Clearly, $\psi(\tau)=O(\tau^{-2}e^{-2\tau \delta})$.
Then the left-hand side in (\ref{Okt18a}) is equal to
\begin{equation}\label{M7ac}
-\Big(\frac{\pi}{\kappa\sinh(\pi\kappa}\Big)^{1/2}\sqrt{1+A^2}\sin(\kappa\log(\frac{1}{2}\tau r)-\gamma_\kappa+\psi(\tau))
\end{equation}
and equation (\ref{Okt18a}) can be written as
$$
\sin\big(\kappa\log\big(\frac{1}{2}\tau r\big)-\gamma_\kappa+\psi\big)= \sin\big(\kappa\log\big(\frac{1}{2}r\big)+\gamma\big),
$$
which implies
\begin{equation}\label{Okt18b}
\kappa\log
\tau=\gamma_\kappa+\gamma-\psi+k\pi,
\end{equation}
where $k$ is a large positive integer. Thus
\begin{equation*}
\tau=e^{(\gamma_\kappa+\gamma+k\pi)/\kappa}\Big(1+O(e^{-2\tau \delta})\Big).
\end{equation*}
We denote this eigenvalue by $\widehat{\tau}_k$. It is defined for $k\geq k_\delta$, where $k_\delta$ is a sufficiently large integer depending on $\delta$. Then
\begin{equation}\label{Okt18c}
\widehat{\tau}_k=e^{(\gamma_\kappa+\gamma+k\pi)/\kappa}\Big(1+O\Big(e^{-2\delta\exp(e^{(\gamma_\kappa+\gamma+k\pi)/\kappa}) }\Big)\Big).
\end{equation}
\bigskip

Let us formulate this result as

\begin{proposition} There exists an interger $k_\delta$ depending on $\delta$ such that the eigenvalues of the operator (\ref{F9a})  are simple and exhausted by (\ref{Okt18c}).
The corresponding eigenfunction is given by
\begin{equation}\label{Okt18cz}
\Phi(r)=\Phi_k(r)=K_{i\kappa}(\tau r)+Q(\tau)I_{i\kappa}(\tau r),\;\;\mbox{where $\tau=\widehat{\tau}_k$}.
\end{equation}

\end{proposition}

We note also that
\begin{equation}\label{M3c}
\int_0^\delta \Phi^2(r)r^{1+2\beta}dr\sim\tau^{-2}\int_0^{\delta\tau}|K_{i\kappa}(s)|^2s^{1+2\beta}ds\sim\tau^{-2-2\beta}
\;\;\mbox{for $\beta>-1$},
\end{equation}
\begin{equation}\label{M3ca}
\int_0^\delta |\partial_r\Phi|^2(r)r^{1+2\beta}dr\sim \tau^{-2\beta}\;\;\mbox{for $\beta>0$}
\end{equation}
and
\begin{equation}\label{M3cb}
\int_0^\delta |\partial^2_r\Phi|^2(r)r^{1+2\beta}dr\sim c\tau^{2-2\beta},\;\;\mbox{for $\beta>1$}.
\end{equation}
Moreover,
$$
h(\delta)=O(\tau^{-1/2}e^{-\delta\tau}).
$$

\subsection{Two lemmas}

Here we obtain some estimates for solutions to the problem
\begin{eqnarray}\label{K191}
&&(-\Delta+\tau^2) u=f\;\;\mbox{in $A$}\nonumber\\
&&r^{-1}(\partial_\theta -\rho_0)u=g\;\;\mbox{for $\theta=\alpha^*$}\nonumber\\
&&\partial_\theta u=0\;\;\mbox{for $\theta=0$}
.
\end{eqnarray}

Let $W^2_\beta (A)= V^2_\beta (A)\bigcap L^2_\beta(A).$
We supply it with the norm
$$
||v||_{W^2_\beta (A)}=||v||_{V^2_\beta (A)}+||v||_{L^2_\beta(A)}.
$$
We introduce also the spaces
\begin{equation}\label{J26ca1}
{\mathcal D}_\beta=\{u=C\zeta(\tau r)\phi_0(\theta)\sin(\kappa\log \frac{1}{2}r+\gamma)+v\;:\;v\in W^2_\beta (A),\;\partial_\theta v=0\;\; \mbox{for}\;\; \theta=0\},
\end{equation}
which will be used for $\beta<1$. In this case it differs from $W^2_\beta (A)$.
Let also $W^{1/2}_\beta (0,\infty)$ is the space of functions on $(0,\infty)$ with the norm
\begin{equation}\label{J26ca2}
||g||_{W^{1/2}_\beta (0,\infty)}=||g||_{V^{1/2}_\beta (0,\infty)}+||g||_{L^2_\beta(0,\infty)}
\end{equation}

We will use the following splitting of solutions of (\ref{K191}):
\begin{equation}\label{F20ba}
u=h(r)\phi_0+w,\;\;\int_0^{\alpha^*}w\phi_0d\theta=0\;\;\mbox{for allmost all $r>0$},
\end{equation}
where $\phi_0$ is given by (\ref{F7a}). Clearly
$$
h(r)=\int_0^{\alpha^*}u\phi_0d\theta.
$$
Multiplying the first equation in (\ref{K191}) by $\phi_0$ and integrating over $(0,\alpha^*)$
we get
\begin{equation}\label{F7aa}
({\mathcal M}+\tau^2)h=f_0+\phi(\alpha^*)r^{-1}g\;\;\mbox{on $(0,\infty)$}
\end{equation}
and
\begin{eqnarray}\label{K191a}
&&(-\Delta+\tau^2) w=F:=f-f_0-\phi(\alpha^*)r^{-1}g\phi_0\;\;\mbox{in $A$}\nonumber\\
&&r^{-1}(\partial_\theta -\rho_0)w=g\;\;\mbox{for $\theta=\alpha^*$}\nonumber\\
&&\partial_\theta w=0\;\;\mbox{for $\theta=0$}.
\end{eqnarray}

\begin{lemma}\label{LF20a} Let $\beta\in (1-\mu_1,1)$,   $q_0\in (0,q)$ and $\tau\in [q_0^{-1}\tau_k,q_0\tau_k]$, $\tau\neq\tau_k$. Let also
\begin{equation}\label{KoM2}
u=C\zeta(\tau r)\phi_0(\theta)\sin(\kappa\log \frac{1}{2}r+\gamma)+v,\;\;v\in W^2_\beta (A),\;\;\mbox{$C$ is constant,}
\end{equation}
solve (\ref{K191}), where $f\in L^2_\beta(A)$ and $g\in W^{1/2}_\beta(A)$. Then
\begin{equation}\label{J20b}
\tau^{-2\beta}|C|^2+\tau^{-2}||v||_{V^2_\beta(A)}+\tau^2||v||^2_{L^2_\beta(A)}\leq \frac{c}{|\tau-\tau_{k}|^2}\Big(||f||^2_{L^2_\beta(A)}
+||g||^2_{V^{1/2}_\beta(0,\infty)}+||g||_{L^2_\beta(0,\infty)}\Big).
\end{equation}
Moreover
\begin{equation}\label{F20bb}
C=\Big(\frac{\kappa\sinh(\pi\kappa)}{\pi}\Big)^{1/2}
\frac{1}{\kappa\sin(\kappa\log(\tau_k/\tau))}\int_0^\infty K_{i\kappa}(\tau r)(f_0+\phi(\alpha^*)r^{-1}g)rdr.
\end{equation}
\end{lemma}
\begin{proof} We use the representation (\ref{F20ba}). Then
$$
h(r)=C\zeta(\tau r)\sin(\kappa\log \frac{1}{2}r+\gamma)+\widehat{v},\;\;\widehat{v}\in\widehat{V}^2_\beta(0,\infty).
$$
and
$$
v=w+\widehat{v}\phi_0(\theta).
$$

Let us prove the inequality
\begin{equation}\label{J26bx}
\tau^{-2}||w||^2_{V^2_\beta(A)}+\tau^2||w||^2_{L^2_\beta(A)}\leq c\tau^{-2}(||F||^2_{L^2_\beta(A)}+||g||^2_{V^{1/2}_\beta (0,\infty)}+\tau ||g||^2_{L^2_\beta(0,\infty)}).
\end{equation}
By scaling we can reduce the estimate to the case $\tau=1$. For $\beta=0$ the corresponding quadratic form is positive definite and the proof is standard for the weak solution, after that it is enough to use local estimates\footnote{For local estimate near the origin we used the fact that $0\in (1-\mu_1,1)$, which follows from the assumption (\ref{K5y}).}. Extension to other values of $\beta$ can be done also by using local estimates near the origin and infinity.

Now the combination of the estimates (\ref{J26bx}), (\ref{J26bq}) and (\ref{J26ba}) leads to (\ref{J20b}) and (\ref{F20bb}).
\end{proof}

\begin{lemma}
Let $\beta \in (1-\mu_1,1)$ and let $q_0\in (0,q)$, $\tau\in [q_0^{-1}\tau_k,q_0\tau_k]$. Let also
\begin{equation}\label{F20bc}
\int_0^\infty K_{i\kappa}(\tau_k r)(f_0+\phi(\alpha^*)r^{-1}g)rdr=0.
\end{equation}
Then the solution (\ref{KoM2}) of (\ref{K191}) satisfies
\begin{equation}\label{J26bqq}
\tau^{-2\beta}|C|^2+\tau^{-2}||v||^2_{V^2_\beta (A)}+\tau^2||v||^2_{L^2_\beta(A)}
\leq c\tau^{-2}\Big(||f||^2_{L^2_\beta(A)}+||g||^2_{V^{1/2}_\beta(0,\infty)}+\tau ||g||^2_{L^2_\beta(0,\infty)}\Big),
\end{equation}
where $c$ does not depend on $\tau$, $f$ and $g$.
\end{lemma}

\begin{proof} The proof basically repeats the proof of Lemma \ref{LF20a} but instead of (\ref{J26b}) and (\ref{J26ba})
we must use (\ref{J26bz}) and (\ref{J26bzz})).

\end{proof}

\section{Self-adjoint extensions of the operator ${\mathcal L}$ with the boundary conditions  (\ref{K2bbzz})}\label{Sselfadj}

First consider the equation
\begin{equation}\label{K2bba}
-\Delta u+\sigma u=f\;\;\mbox{in $\Omega$}
\end{equation}
supplied with the boundary conditions
\begin{equation}\label{K2bbq}
\partial_\nu u-r^{-1}\rho u=g
\;\;\mbox{on $S$}
\end{equation}
and
\begin{eqnarray}\label{K2bbb}
&&u=0\;\;\mbox{on $B$}\nonumber\\
&&\partial_xu|_{x=0}=\partial_xu|_{x=\Lambda/2}=0.
\end{eqnarray}
Let $V_\beta^{2}(\Omega)$,  $\beta\in\Bbb R$, be the  space of functions in $\Omega$ with finite norm
$$
||u||_{V^{2}_\beta(\Omega)}=\Big(\int_\Omega\sum_{i+j\leq 2}|\partial_x^i\partial_y^ju|^2r^{2(\beta-2(2-i-j))}dxdy\Big)^{1/2}.
$$
The space $L^2_\beta(\Omega)$ consists of functions $f$ in $\Omega$ with the finite norm
$$
||u||_{L^2_\beta(\Omega)}=\Big(\int_\Omega |f|^2r^{2\beta}dxdy\Big)^{1/2}.
$$
We introduce also the following subspace of $V_\beta^{2}(\Omega)$
$$
\widetilde{V}_\beta^{2}(\Omega)=\{u\in V_\beta^{2}(\Omega)\,:\,u|_B=0,\;u|_{x=0}=0\;\mbox{and}\;u|_{x=\Lambda/2}=0\}.
$$
The space $V^{1/2}_\beta(S)$ consists of functions defined on $S$ and has the norm
$$
||g||_{V^{1/2}_\beta(S)}=\inf\Big\{||u||_{V^{1}_\beta(\Omega)}\,:\,u\in \widetilde{V}^{1}_\beta(\Omega),\, u|_{S}=g\Big\}.
$$
Another equivalent norm is the following (see \cite{MP78})
$$
\Big(\int_0^{\Lambda/2}\int_0^{\Lambda/2}\frac{|s^\beta g(s)-x^\beta g(x)|^2}{|s-x|^2}dsdx+\int_0^{\Lambda/2}x^{2\beta-1}|g|^2dx\Big)^{1/2}.
$$
Here we used the parametrisation  $y=\eta(x)$ on $S$.

We put
$$
{\mathcal L}u=-\Delta u+\sigma u,\;\;{\mathcal B}u=(\partial_\nu-r^{-1}\rho) u|_S.
$$
One can verify that the operator
$$
({\mathcal L},{\mathcal B})\,:\,\widetilde{V}_\beta^{2}(\Omega)\rightarrow L^2_\beta(\Omega)\times V^{1/2}_\beta(S)
$$
is continuous.

Using Proposition \ref{PF22} and well known results from theory of boundary value problems in domains with angular points on the boundary (see \cite{Nazarov1994} or \cite{KMR97}), we get the following assertion

\begin{proposition}\label{P2} {\rm (i)} If $\beta-1\neq 0$ and $\beta+\mu_j-1\neq 0$ for $j=\pm 1,\pm 2,\ldots$ then the operator
$$
({\mathcal L},{\mathcal B}):\widetilde{V}_\beta^{2}(\Omega)\rightarrow L^2_\beta(\Omega)\times V^{1/2}_\beta(S)
$$
is Fredholm.

{\rm (ii)} Let $\beta_j$, $j=1,2$, satisfy $1-\alpha<\beta_1<1<\beta_2<\mu_1+1$
  and let $f\in L^{2}_{\beta_1}(\Omega)$,  $g\in V^{1/2}_{\beta_1}(S)$ and $u_2\in V_{\beta_2}^{2}(\Omega)$ be a solution to {\rm (\ref{K2bba})--(\ref{K2bbb})}. Then
$$
u_2=\zeta(r)(c_+r^{i\kappa }+c_-r^{^-i\kappa })\varphi_0(\kappa\theta)+u_1,
$$
where $u_1\in V^{2}_{\beta_1}(\Omega)$   and $c_{\pm}\in\Bbb C$.

{\rm (iii)} Let $\beta_j$, $j=1,2$, satisfy $1-\mu_1<\beta_2<\beta_1<1 $
  and let $f\in L^{2}_{\beta_2}(\Omega)$, $g\in V^{1/2}_{\beta_2}(S)$ and $u\in V^{2}_{\beta_1}(\Omega)$ be a solution to {\rm (\ref{K2bba})--(\ref{K2bbb})}. Then
 $u\in V^{2}_{\beta_2}(\Omega)$.
\end{proposition}

One can verify that the operator ${\mathcal L}$ is symmetric on
$$
D_0=\{u\in\widetilde{V}_0^{2}(\Omega)\,:\,{\mathcal B}u=0\}.
$$
 To obtain "real valued", self-adjoint extensions of this operator we proceed as follows. We choose $\gamma\in [0,\pi)$ and put
 \begin{equation}\label{F22aa}
 w_\gamma=\sin(\kappa\log\frac{1}{2}r+\gamma)\phi_0,
 \end{equation}
Let $\widehat{w}$ be  such that
$$
{\mathcal L}\widehat{w}\in L^2(\Omega)\;\;\mbox{and}\;\; {\mathcal B}\widehat{w}=0\;\;\mbox{on} \; S
$$
 and
 $$
 \widehat{w}=\zeta(r)w_\gamma+w,\;\;w\in \widetilde{V}_{\beta_*}^{2}(\Omega)\;\; \mbox{with a certain $\beta_*\in (1-\alpha,1)$.}
 $$
Since the function $w$ satisfies (\ref{K2bba})-(\ref{K2bbb}) with
$$
f=-({\mathcal L}+\tau^2)\zeta(\tau r) w_\gamma,\;\;\;g=-{\mathcal B}w_\gamma,
$$
 using properties of functions $\eta$, $\rho$ and $\sigma$, we get that $f\in L_{\beta_*}(\Omega)$ and $g\in V^{1/2,2}_{\beta_*}(S)$ with any $\beta_*\in (1-\alpha,1)$ and the existence of such $w$  follows from Proposition \ref{P2}(i). Moreover $w\in \tilde{V}_{\beta_*}^{2}(\Omega)$ for any $\beta_*\in (1-\alpha,1)$.
We define a domain of ${\mathcal L}$ as
\begin{equation}\label{J20a}
{\mathcal D}_\gamma=\{u=a\widehat{w}+v\,:\,a\in\Bbb C,\;v\in D_0,\;{\mathcal B}u=0\}
\end{equation}
In the next proposition we will show that this definition does not depend on the choice of $\widehat{w}$ and $\zeta$ and determines only by $w_\gamma$ and that the operator ${\mathcal L}$ with the domain ${\mathcal D}_\gamma$ is self-adjoint.

\begin{proposition}\label{P22aa} (i) There exists a function $\widehat{w}$ introduced above. The domain ${\mathcal D}_\gamma$ does not depend on the choice of $\widehat{w}$ and cut-off function $\zeta$.

(ii) The operator ${\mathcal L}$ defined on the domain ${\mathcal D}_\gamma$ is self-adjoint.

\end{proposition}
\begin{proof} (i) The existence of such $\widehat{w}$ we have proved above the proposition. If we have two such functions $\widehat{w}_1$ and
$\widehat{w}_2$ then the difference $W=\widehat{w}_1-\widehat{w}_2\in \widetilde{V}_{\beta_*}^{2}(\Omega)$  and $W$ satisfies (\ref{K2bba})-(\ref{K2bbb}) with $f\in L^2(\Omega)$ and $g=0$.

Applying Proposition \ref{P2} (iii) with $\beta_1=\beta_*$ and  $\beta_2=2$ and using that $\mu_1>1$, we obtain $W\in V^{2}_2(\Omega)\subset L^2(\Omega)$, which proves the result.

(ii) Let $B_\varepsilon$, where $\varepsilon$ is a small positive number, be the ball of radius $\varepsilon$ centered at $(0,R)$. We put
$$
\Omega_\varepsilon=\Omega\setminus B_\varepsilon,\;\;S_\varepsilon=S\setminus B_\varepsilon.
$$
Let also $\gamma_\varepsilon=\{(x,y)\in\Omega\;:\,r=\varepsilon\}$. Let $U_k=a_k\widehat{w}+v_k\in D_\gamma$. 
Then
\begin{eqnarray}\label{Sept23a}
&&\int_{\Omega_\varepsilon}\Big({\mathcal L} U_1\overline{U_2}-U_1{\mathcal L}\overline{U_2}\Big)dxdy\nonumber\\
&&=\int_{\gamma_\varepsilon}\Big(\partial_rU_1\overline{U_2}-U_1\partial_r\overline{U_2}\Big)rd\theta\rightarrow
a_1\overline{a_2}\lim_{r\to 0}\int_{-\pi/3}^{\pi/3}\Big(\partial_r\widehat{w}\overline{\widehat{w}}-\widehat{w}\partial_r\overline{\widehat{w}}\Big)rd\theta=0.
\end{eqnarray}
This shows that the operator ${\mathcal L}$ is symmetric on the domain ${\mathcal D}_\gamma$. We denote this operator by ${\mathcal L}_\gamma$. Consider the adjoint to ${\mathcal L}_\gamma$ operator
$$
{\mathcal L}_\gamma^*:L^2(\Omega)\rightarrow {\mathcal D}_{\gamma}^*.
$$
In order to prove our proposition it is sufficient to show that ${\mathcal L}_\gamma^*u=f\in L^2(\Omega)$ implies $u\in {\mathcal D}_{\gamma}$ and ${\mathcal L}_\gamma u=f$.
  Assume that ${\mathcal L}_\gamma^*u=f\in L^2(\Omega)$. Then
$$
\int_\Omega \Big({\mathcal L} U\overline{u}-U\overline{f}\Big)dxdy=0\;\;\mbox{$\forall U\in {\mathcal D}_\gamma$}.
$$
Using local estimates for elliptic boundary valued problems one can show that $u\in V^{2}_\beta(\Omega)$ with $\beta=2$ since $u\in L^2(\Omega)$. Using that $f\in L^2(\Omega)$ together with Proposition \ref{P2}(ii), we get that $u=\zeta(r)w+v$, where $v\in V^{2}_0(\Omega)$ and $w\in {\mathcal X}$. Since the form $q$ is non-generating on ${\mathcal X}$, calculations similar to (\ref{Sept23a}) show that $w$ must be proportional to $w_\gamma$. This proves that the operator ${\mathcal L}$ defined on the domain ${\mathcal D}_\gamma$ is self-adjoint.
\end{proof}

As a consequence of the above result we get the following observation
\begin{corollary}\label{cor1} the operator
\begin{equation}\label{Okt6a}
{\mathcal L}:{\mathcal D}_\gamma\rightarrow L^2(\Omega)
\end{equation}
is Fredholm with the index $0$. (In what follows we denote the operator (\ref{Okt6a}) by ${\mathcal L}_\gamma$.)
\end{corollary}

Since the inclusion $D_\gamma\to L^2(\Omega)$ is compact and the operator (\ref{Okt6a}) is Fredholm with index $0$ (according to Corollary \ref{cor1}), the spectrum of ${\mathcal L}_\gamma$ consists of isolated eigenvalues of finite multiplicities with possibly accumulation points at $\pm\infty$. Thus we get

\begin{proposition}\label{PM10}
The spectrum of the operator ${\mathcal L}_\gamma$ consists of eigenvalues of finite multiplicity with the accumulating points at $+\infty$ and  $-\infty$.
\end{proposition}

Since every element  $u\in{\mathcal D}_\gamma$ admits a unique representation $u=c\widehat{w}+v$, where $v\in V^{2}_0(\Omega)$ and $c$ is a constant we define the norm in ${\mathcal D}_\gamma$ as
\begin{equation}\label{J1a}
||u||_{{\mathcal D}_\gamma}=\Big(||u||_{V^{2}_0(\Omega)}^2+|c|^2\Big)^{1/2}.
\end{equation}

\section{Proof of Theorem \ref{Th1.1}}\label{Sproof}

By Propositions \ref{P22aa}(ii) and \ref{PM10} it remains to prove the asymptotics (\ref{K49s}) and simplicity of large negetive eigenvalues in Theorem \ref{Th1.1}.
The proof  consists of several steps. First by a suitable change of variables we represent the problem as as a small, in a certain sense, perturbation of a problem whose negative eigenvalues  can be analyzed explicitly. The important property of the unperturbed problem is the fact that the distance between neighbor eigenvalues is comparable with the absolute value of corresponding eigenvalues. A specific of this representation consists of its dependence on a certain parameter $\delta$ and the perturbation analysis involves careful study of dependence of this perturbation analysis on this parameter. An additional complexity is brought by possibli different domains of perturbed and unperturbed operators. This can be overcome by extending domains of this operators by using weights and observation that the eigenvalues are preserved  for both operators.

\subsection{Change of variables}

We choose  functions $\xi=\xi(x)$ and $\chi=\chi(x)$ defined on $[0,\Lambda/2]$ and belonging to $C^{2}([0,\Lambda/2])$ and
$C^{1}([0,\Lambda/2])$ respectively, and subject to the following properties:

(i) $\xi(x)=\eta(0)-a_0x$ and $\chi(x)=\rho_0/R$ for $x\in(0,3\delta)$, where $R=(x^2+(\xi(x)-\xi(0))^2)^{1/2}$;

(ii) $\xi(x)=\eta(x)$ and $\chi(x)=\rho(x)$ for $x\in [\Lambda/2-\delta,\Lambda/2]$;

(iii)
$$
|\xi(x)-\eta(x)|+|\xi'(x)-\eta'(x)|\leq c\delta^\alpha\;\;\mbox{for $x\in [0,\Lambda/2]$}
$$
 and
$$
|\chi(x)-\rho(x)|\leq c\delta^\alpha\;\;\mbox{for $x\in [0,\Lambda/2]$}.
$$
Here $c$ is a certain constant independent of $\delta$ and $\delta$ is a certain positive number, which will be chosen later.

\bigskip
Let us make the following change of variables:
$$
X=x,\;\;Y=\frac{y\xi(x)}{\eta(x)}\;\;\mbox{and}\;\; U(X,Y)=u(X,Y\frac{\eta(x)}{\xi(x)}).
$$
Then
$$
\partial_x=\partial_X+Y\Big(\frac{\xi'}{\xi}-\frac{\eta'}{\eta}\Big)\partial_Y,\;\;\partial_y=\frac{\xi}{\eta}\partial_Y.
$$
Therefore the problem (\ref{K2bb}) becomes
\begin{eqnarray*}
&&LU:=-\Big(\partial_X+Y\Big(\frac{\xi'}{\xi}-\frac{\eta'}{\eta}\Big)\partial_Y\Big)
\Big(\partial_XU+Y\Big(\frac{\xi'}{\xi}-\frac{\eta'}{\eta}\Big)\partial_YU\Big)-\frac{\xi^2}{\eta^2}\partial_Y^2U+\sigma U=\lambda U,\\
&&BU:=\frac{(-\eta',1)}{\sqrt{1+\eta'^2}}\Big(\partial_XU+Y\Big(\frac{\xi'}{\xi}-\frac{\eta'}{\eta}\Big)\partial_YU,
\frac{\xi}{\eta}\partial_YU\Big)-r^{-1}\rho U=0
\end{eqnarray*}
and
\begin{equation}\label{Okt21b1}
U(X,0)=0,\;\;\partial_XU|_{X=0}=0,\;\;\partial_XU|_{X=\Lambda/2}=0
\end{equation}

We represent the operators ${\mathcal L}$ and ${\mathcal B}$ as
\begin{equation}\label{F24aa}
{\mathcal L}={\mathcal L}_0+{\mathcal L}_1\;\;\mbox{ and}\;\; {\mathcal B}={\mathcal B}_0+{\mathcal B}_1,
\end{equation}
 where
$$
{\mathcal L}_0U=-\Delta U,\;\;{\mathcal B}_0U=\partial_\nu U-r^{-1}\chi U,
$$
\begin{eqnarray*}
&&{\mathcal L}_1U=-Y\Big(\frac{\xi'}{\xi}-\frac{\eta'}{\eta}\Big)\partial_Y
\Big(\partial_XU+Y\Big(\frac{\xi'}{\xi}-\frac{\eta'}{\eta}\Big)\partial_YU\Big)\\
&&-\Big(\partial_X+Y\Big(\frac{\xi'}{\xi}
-\frac{\eta'}{\eta}\Big)\partial_Y\Big)
Y\Big(\frac{\xi'}{\xi}-\frac{\eta'}{\eta}\Big)\partial_YU-\Big(\frac{\xi^2}{\eta^2}-1\Big)\partial_Y^2U+\sigma U
\end{eqnarray*}
and
\begin{eqnarray*}
&&{\mathcal B}_1U=\frac{(-\eta',1)}{\sqrt{1+\eta'^2}}\Big(Y\Big(\frac{\xi'}{\xi}-\frac{\eta'}{\eta}\Big)\partial_YU,
\Big(\frac{\xi}{\eta}-1\Big)\partial_YU\Big)\\
&&+\Big(\frac{(-\eta',1)}{\sqrt{1+\eta'^2}}-\frac{(-\xi',1)}{\sqrt{1+\xi'^2}}\Big)\big(\partial_XU,
\partial_YU\big)-r^{-1}(\rho-\chi)U
\end{eqnarray*}

We note that
\begin{equation}\label{Nov1ab}
\frac{\xi}{\eta}-1=O(x^{1+\alpha}),\;\;
\frac{\xi'}{\xi}-\frac{\eta'}{\eta}=O(x^{\alpha}), \;\;\frac{d}{dx}\Big(\frac{\xi'}{\xi}-\frac{\eta'}{\eta}\Big)=O(x^{\alpha-1})\;\;\mbox{near $x=0$}
\end{equation}
and
\begin{equation}\label{F16a}
\Big|\frac{\xi}{\eta}-1\Big|\leq c\delta^\alpha,\;\;\Big|\frac{\xi'}{\xi}-\frac{\eta'}{\eta}\Big|\leq c\delta^\alpha\;\;\mbox{on $[0,\Lambda/2]$}.
\end{equation}

In what follows it will be important for us that
\begin{equation}\label{F16aa}
|{\mathcal L}_1U|\leq c\Big(x^\alpha (|\partial_X^2U|+|\partial_Y^2U|+|\partial_{XY}^2U|)+x^{\alpha-1}(|\partial_XU|+|\partial_YU|)
+|U|\Big),
\end{equation}
\begin{equation}\label{F16ab}
|{\mathcal B}_1U|\leq c\Big(x^\alpha (|\partial_XU|+|\partial_YU|)+x^\alpha r^{-1}|U|\Big)
\end{equation}
for small $x$ and
\begin{equation}\label{F16ac}
|{\mathcal L}_1U|\leq c\delta^\alpha (|\partial_X^2U|+|\partial_Y^2U|+|\partial_{XY}^2U|)+C_\delta(|\nabla U|+|U|),
\end{equation}
\begin{equation}\label{F22ac}
|{\mathcal B}_1U|\leq c\delta^\alpha (|\nabla U|+R^{-1}|U|)
\end{equation}
on $[0,\Lambda/2]$.

Now the problem can be written as
\begin{eqnarray}\label{F15b}
&& ({\mathcal L}_0+{\mathcal L}_1)U+\tau^2U=0\;\;\mbox{in $\Omega_\xi$}\nonumber\\
&&({\mathcal B}_0+{\mathcal B}_1)U=0\;\;\mbox{on $S_\xi$}
\end{eqnarray}
\begin{equation}\label{F8aaa}
\partial_XU|_{X=0}=0,\;\;\partial_XU|_{X=\Lambda/2}=0\;\;\mbox{and}\;\;U|_{Y=0}=0.
\end{equation}

In forthcoming analysis we will consider the problem (\ref{F15b}), (\ref{F8aaa}) as a perturbation of the same problem with ${\mathcal L}_1=0$ and ${\mathcal B}_1=0$. It will be important for us to control dependence on $\delta$ and $\tau$ in constants in our consideration. We will indicate this dependence by putting indexes, for example $c_\delta$.

\subsection{Operator $({\mathcal L}_0, {\mathcal B}_0)$}\label{SM10a}

First let us consider the unperturbed problem
\begin{eqnarray}\label{F8a}
&& {\mathcal L}_0U+\tau^2U=F\;\;\mbox{in $\Omega_\xi$}\nonumber\\
&&{\mathcal B}_0U=G\;\;\mbox{on $S_\xi$}
\end{eqnarray}
\begin{equation}\label{F8aa}
U(X,0)=0,\;\;\partial_XU|_{X=0}=0,\;\;\partial_XU|_{X=\Lambda/2}=0.
\end{equation}

Define the spaces: $V^l_\beta(\Omega_\xi)$ consists of functions on $\Omega_\xi$ satisfying (\ref{F8aa}) and having the finite norm
$$
||v||_{V^l_\beta(\Omega_\xi)}=\Big(\sum_{j+k\leq 2}\int_{\Omega_\xi}R^{2(\beta+l-j-k)}|\partial_X^j\partial_Y^kv|^2dXdY\Big)^{1/2}
$$
and $L_\beta^2(\Omega_\xi)$ has the norm
$$
||f||_{L^2_\beta(\Omega_\xi)}=\Big(\int_{\Omega_\xi)}R^{2\beta}|f|^2dXdY\Big)^{1/2}.
$$
Let also $V^{1/2}_\beta(S_\xi)$ consists of traces on $S_\xi$ of functions from $V^1_\beta(\Omega_\xi)$. We parameterize $S_\xi$ by $Y=\xi(X)$, $X\in (0,\Lambda/2)$ and use the following norm there
\begin{equation}\label{F19a}
||g||_{V^{1/2}_\beta(S_\xi)}=\Big(\int_0^{\Lambda/2}\int_0^{\Lambda/2}\frac{|y^\beta g(y)-x^\beta g(x)|^2}{|y-x|^2}dx\Big)^{1/2}+||g||_{L^2_{\beta-1/2}(0,\Lambda/2)},
\end{equation}
where $g(x)$ is a function on the boundary $S_\xi =\{(x,\xi(x))\,:\,x\in (0,\Lambda/2)\}$.

In the forthcoming analysis an important role will play a certain splitting of the boundary value problem into a problem for an ODE on an interval and a boundary value problem which has a positivity property. Let us describe this splitting.

Multiplying the first equation in (\ref{F8a}) by $\phi_0$ and integrating over the interval $(0,\alpha^*)$ with respect to $\theta$, we get for $r<3\delta$ (here and in what follows we use the notation $r$ instead of $R$ in order to emphasize the similarity in equations below and in Sect.\ref{Sec11}.
\begin{equation}\label{F12c}
-\frac{1}{r}\partial_rr\partial_rh(r)-\frac{1}{r^2}\int_0^{\alpha^*}\partial_\theta^2U\phi_0d\theta+\tau^2h(r)=f(r),
\end{equation}
where
\begin{equation}\label{F13b}
h(r)=\int_0^{\alpha^*}U(r,\theta)\phi_0(\theta)d\theta,\;\;f(r)=\int_0^{\alpha^*}F(r,\theta)\phi_0(\theta)d\theta.
\end{equation}
Integrating by parts in the integral in (\ref{F12c}) and using boundary conditions for $U$ and $\phi_0$, we obtain
\begin{equation}\label{F12ca}
{\mathcal M}h(r)+\tau^2h(r)=f(r)+\frac{1}{r}G(r)\phi_0(\alpha^*)\;\;\mbox{for $r<3\delta$}.
\end{equation}
Representing the function $U$ in the form
\begin{equation}\label{F12cba}
U=h(r)\phi_0+W,\;\;\int_0^{\alpha^*}W\phi_0d\theta=0\;\;\mbox{for $r<\delta$},
\end{equation}
we have that
\begin{equation}\label{F12cb}
{\mathcal L}_0W+\tau^2W=F_1:=F-f(r)-\frac{1}{r}G(r)\phi_0(\alpha^*)
\end{equation}
and
\begin{equation}\label{F12cc}
{\mathcal B}_0W=G.
\end{equation}

Introduce the space $Y_0$ as the closure of functions $U$ satisfying (\ref{F8aa}) and
$$
\int_0^{\alpha^*}U\phi_0d\theta=0\;\;\mbox{for $r<\delta$}.
$$
in the norm
$$
||U||_{Y_0}=\Big(\int_{\Omega_\xi}(|\nabla U|^2+|U|^2)dXdY\Big)^{1/2}.
$$


We define the bilinear form on $Y_0$
$$
a(U_1,U_2)=\int_{\Omega_\xi}(\nabla U_1\cdot\nabla U_2+\tau^2U_1U_2)dXdY-\int_{S_\xi}r^{-1}\chi U_1U_2ds.
$$
Here the first integral is positive and the last one is negative. For every positive $\varepsilon$ we have
\begin{equation}\label{F24b}
\int_0^{\Lambda/2}U^2(x,\xi(x))dx\leq \int_{\Omega_\xi}(\varepsilon |U_Y|^2+\varepsilon^{-1} |U|^2)dXdY.
\end{equation}
Using this estimate, one can verify the following assertion
\begin{lemma}\label{LM7} For $U\in Y_0$ we have
\begin{equation}\label{F12cc}
a(U,U)\geq c\int_{\Omega_\xi}(|\nabla U|^2dXdY+(\tau^2-C_\delta)\int_{\Omega_\xi}|U|^2)dXdY.
\end{equation}
\end{lemma}
This implies in particular that for $\tau^2\geq C_\delta +1$ the form $a$ is positive definite.

\bigskip

\begin{lemma}\label{LF8a}
Let  $\tau^2\geq C_\delta +1$, where $C_\delta$ is the constant in (\ref{F12cc}). Let also $F\in L^2(\Omega_\xi)$ and $G\in L^2(S_\xi)$.
Then the problem (\ref{F8a}), (\ref{F8aa}) has a unique weak solution $U\in Y_0$ and this solution satisfies the estimate
\begin{equation}\label{J26bxx}
\int_{\Omega_\xi}(|\nabla U|^2+\tau^2|U|^2)dXdY\leq c(\tau^{-2}||F||^2_{L^2(\Omega_\xi)}+\tau^{-1}||G||^2_{L^2(S_\xi}),
\end{equation}
where $c$ does not depend on $\tau$, $F$ and $G$.
\end{lemma}
\begin{proof}  Multiplying the first equation in (\ref{F8a}) by $V\in Y_0$ and integrating over $\Omega_\xi$ we get
$$
a(U,V)=\int_{\Omega_\xi}FVdXdY+\int_{S_\xi}GVds\;\;\mbox{for all $V\in Y_0$},
$$
which represents a weak formulation of the problem (\ref{F8a}), (\ref{F8aa}) in $Y_0$.
Since the form $a$ is positive definite for large $\tau$ the above weak formulation has a unique solution $U\in Y_0$.
 Using the estimate (\ref{F24b}) one can show that $U$ satisfies (\ref{J26bxx}).
\end{proof}

The space $Y$ consists of functions
$$
U=b\zeta(r-\delta)e^{-\tau(r-\delta)}\phi_0+W,\;\;\mbox{$b$ is a constant and $V\in Y_0$},
$$
where $\zeta(t)$ is a cut-off function equal $1$ for $t<\delta/2$ and $0$ for $t>\delta$). The function
$$
h(r)=\int_0^{\alpha^*}U\phi_0d\theta
$$
is well defined for $r\in [\delta,3\delta)$.
\begin{lemma}\label{LF13a} There exists $C_\delta$ such that if $\tau\geq C_\delta$ then the problem
\begin{equation}\label{F13a}
a(U,V)=0\;\;\mbox{for all $V\in Y_0$}
\end{equation}
has a unique solution in $Y$ satisfying $h(\delta)=b$. Moreover,
\begin{equation}\label{F13aa}
U=b\big(\phi_0(\theta)\zeta(r)\frac{K_{i\kappa}(\tau r)}{K_{i\kappa}(\tau\delta )}+W\big),
\end{equation}
where the function $W\in Y_0$ satisfies the estimate
\begin{equation}\label{F13ab}
||W||_{Y_0}\leq c|b|e^{-\tau\delta}
\end{equation}
\end{lemma}
\begin{proof} Inserting (\ref{F13aa}) into (\ref{F13a}), we obtain the relation
$$
a(W,V)={\mathcal F}(V)\;\;\mbox{for all $V\in Y_0$},
$$
where
$$
F(V)=\frac{b}{K_{i\kappa}(\tau\delta )}\int_0^{2\delta}\int_0^{\alpha^*}(-\zeta^{''}K_{i\kappa}-2\zeta'\partial_rK_{i\kappa})\phi_0VdXdY.
$$
Using the asymptotics of $K_{i\kappa}(z)$ for large $z$ and Lemma \ref{LF8a}, we obtain
$$
|F(V)|\leq c|b|e^{-\tau\delta}||V||_{Y_0},
$$
which leads to (\ref{F13ab}).
\end{proof}

\begin{remark}\label{RF14a}By Lemma (\ref{LF13a}) we can evaluate the normal derivative of the function (\ref{F13b}) at $r=\delta$:
\begin{equation}\label{F13ba}
h'(r)=(-\tau +\alpha(\tau^{-1}))h(\delta)\;\;\mbox{for large $\tau$},
\end{equation}
where $\alpha(z)$ is a $C^\infty$ function in a neighborhood of the origin. If we consider the function $h$ on the interval $(0,\delta)$ then it must satisfy the equation
\begin{equation}\label{F14bb}
{\mathcal M}h(r)+\tau^2h(r)=0\;\;\mbox{for $r\in (0,\delta)$}
\end{equation}
and  the boundary condition(\ref{F13ba}). This allows us to split solutions of (\ref{F8a}), (\ref{F8aa}) with $F=0$ and $G=0$  in to two parts. The first one ($h$-component of $U$) solves the problem (\ref{F14bb}), (\ref{F13ba}) on the interval $(0,\delta)$ and the second one ($W$-component of $U$) belongs to $Y$, solves the problem (\ref{F13a}) and satisfies
$$
\int_0^{\alpha^*}U\phi_0d\theta=h(\delta)\;\;\mbox{for $R=\delta$}.
$$
\end{remark}

\bigskip
Let us take an arbitrary function $U\in V^2_\beta(\Omega_\xi)$, $\beta\in (1,1+\mu_1)$, satisfying the problem (\ref{F8a}), (\ref{F8aa}) with $F=0$ and $G=0$. Let us show that such solutions can be parameterized by the constant $h(\delta)$ for large $\tau$, where $h$ is defined by (\ref{F13b}). Indeed, if $h(\delta)=0$ then by Lemma \ref{LM7} the $W$ component of $U$ vanishes and for $h$ component of $U$ we obtain the Cauchy problem for the operator ${\mathcal M}+\tau^2$. Therefore $h=0$. To show the existence we start from the $h$ component of $U$ and we borrow  it from Sect.\ref{SecM7}. Let $h(\delta)=b$. We take
$$
h(r)=C_1\Phi(r),\;\;\Phi(r)=(K_{i\kappa}(\tau r)+Q(\tau)I_{i\kappa}(\tau r)),\;\
$$
By (\ref{F28aa}) and (\ref{F15a})
\begin{equation}\label{M7ab}
\Phi(\delta)=e^{-\tau\delta}\Big(\frac{\pi}{2\tau\delta}\Big)^{1/2}\Big(1+O(\frac{1}{\tau}\Big)\neq 0
\end{equation}
for large $\tau$ and we choose
$$
C_1=b/\Phi(\delta).
$$
Now solving the problem (\ref{F13a}) with $h(\delta)=b$ we can find the $W$-component of $U$. Since the $h$ component of $U$ and the $h$ component of $W$ have the same Dirichlet and Robin boundary condition $h$ is $C^\infty$ in a neighborhood of $r=\delta$.

We shall denote the solution of (\ref{F8a}), (\ref{F8aa}) with $h$-component
\begin{equation}\label{M7a}
h(r)=(K_{i\kappa}(\tau r)+Q(\tau)I_{i\kappa}(\tau r)),
\end{equation}
by $\widehat{U}_\tau$ and note that by (\ref{F13ab}) and (\ref{M7ab}) its $W$ component satisfies
\begin{equation}\label{M7aa}
||W_\tau||_{Y_0}\leq c_\delta e^{-2\tau\delta}\frac{1}{\tau^{1/2}}.
\end{equation}
Moreover, by (\ref{M7ac})
\begin{equation}\label{M7ad}
K_{i\kappa}(\tau r)-Q(\tau) I_{i\kappa}(\tau r)=-\Big(\frac{\pi}{\kappa\sinh(\pi\kappa)}\Big)^{1/2}\sqrt{1+A(\tau)^2}\sin(\kappa\log(\frac{1}{2}\tau)-\gamma_\kappa+\psi(\tau))
\end{equation}
where $\psi(\tau)$ is given by (\ref{F15ab}) and $A(\tau)$ by (\ref{F15aa}).

\subsection{Spectral problem for the unperturbed operator}\label{SM7aa}

Let us consider the problem
\begin{eqnarray}\label{Okt30a}
&&{\mathcal L}_0U+\tau^2 U\;\;\mbox{in $\Omega_\xi$}\nonumber\\
&&{\mathcal B}_0U=0\;\;\mbox{on $S_\xi$}
\end{eqnarray}
and
\begin{equation}\label{F14a}
U(X,0)=0,\;\;\partial_XU|_{X=0}=0,\;\;\partial_XU|_{X=\Lambda/2}=0.
\end{equation}

 We choose a smooth cut-off function $\zeta(r)$ which is equal to $1$ for $r<\delta/3$ and $0$ for $r>\delta/2$

To describe a self-adjoint operator associated with ${\mathcal L}_0$ we introduce the space
$$
{\mathcal D}_\gamma=\{U=C\zeta(\tau R)\sin(\kappa\log\frac{1}{2}R+\gamma)\cosh(\kappa\theta)+V:\,V\in V_0^{2}(\Omega_\xi),\;\mbox{$V$ subject to (\ref{F14a}) and ${\mathcal B}_0U=0$}\}.
$$

Similar to Theorem \ref{P22aa} one can show that the operator ${\mathcal L}_0$ with the domain ${\mathcal D}_\gamma$ is self-adjoint.


\begin{theorem}\label{TF9a} There exists an integer $k_\delta$ depending on $\delta$ such that the spectrum the operator ${\mathcal L}_0$ with the domain ${\mathcal D}_\gamma$ in $\tau\geq\widehat{\tau}_{k_\delta}$ consists of eigenvalues $\widehat{\tau}_k$, $k\geq k_\delta$, where $\widehat{\tau}_k$ is given by (\ref{Okt18c}). If we denote by $\widehat{U}_k$ the corresponding eigenfunction then
  the corresponding $h$-component of  $\widehat{U}_k$ is equal to the function $\Phi_k$ given by (\ref{Okt18cz}) and the $W$-component of  $\widehat{U}_k$ admits the estimate
\begin{equation}\label{F14bd}
||W||_{Y_0}\leq ce^{-2\delta \tau}.
\end{equation}
\end{theorem}

\begin{proof} The proof follows from the decomposition of solutions to the  problem (\ref{SM10a}) with $F=0$ and $G=0$
given at the end of Sect. \ref{SM10a} (just after Remark \ref{RF14a}.

\end{proof}


The function $U$ with $h$ component $h_\tau(r)=K_{i\kappa}(\tau r)+Q(\tau)I_{i\kappa}(\tau r)$
is well defined for large $\tau$ but it is not necessary has the right asymptotics at zero. We denote this function by $\widehat{U}_\tau$. Then the corresponding $W$-component of $\widehat{U}_\tau$, which will be denoted by $\widehat{V}_\tau$ satisfies (\ref{F14bd}) still.

Denote
\begin{equation}\label{M3b}
d^2_k=\widehat{\tau}_k^2 ||\widehat{U}_k||^2_{L^2(\Omega_\xi)}=\int_0^\infty |K_{i\kappa}(r)|^2rdr+O(e^{-\tau\delta})
\end{equation}
and let
\begin{equation}\label{M3ba}
{\bf U}_k=\widehat{\tau}_kd_k^{-1}\widehat{U}_k.
\end{equation}
Clearly the $L^2$-norm of these eigenfunctions is equal to $1$. Using relations (\ref{M3c})-(\ref{M3cb}), we get
\begin{equation}\label{M3c1}
\int_{\Omega_\xi}r^{2\beta}|{\bf U}_k|^2dXdY\approx \widehat{\tau}_k^{-2\beta}
\;\;\mbox{for $\beta>-1$},
\end{equation}
\begin{equation}\label{M3ca1}
\int_{\Omega_\xi}r^{2\beta}|\nabla{\bf U}_k|^2dXdY\approx \widehat{\tau}_k^{2-2\beta}\;\;\mbox{for $\beta>0$}
\end{equation}
and
\begin{equation}\label{M3cb1}
\int_{\Omega_\xi}r^{2\beta}|\nabla^2{\bf U}_k|^2dXdY\approx \widehat{\tau}_k^{4-2\beta}\;\;\mbox{for $\beta>1$}.
\end{equation}


\subsection{Some estimates}\label{SKKa}

In this section we consider the non-homogeneous problem (\ref{F8a}), (\ref{F8aa}).

\begin{lemma}\label{TF9aa} Let $q_0\in (0,q)$. Assume that $F\in L^2(\Omega_\xi)$ and $G\in V^{1/2}(S_\xi)$ and that $\tau\in[q_0^{-1}\widehat{\tau}_{k},q_0\widehat{\tau}_{k}]$ and $\tau\neq \widehat{\tau}_{k}$, where $k\geq k_\delta$ and $k_\delta$ is an integer depending on $\delta$.
Then there exists a unique solution of  the problem {\rm (\ref{F8a})}, {\rm (\ref{F8aa})}
\begin{equation}\label{F14az}
U=C\zeta(\tau R)\sin(\kappa\log \frac{1}{2}R+\gamma)\phi_0+V,\;\;V\in V^2_\beta (\Omega_\xi).
\end{equation}
Moreover this solution satisfies
\begin{eqnarray}\label{F9bx}
&&|C|^2+\tau^{-2}||V||^2_{V^2_0(\Omega_\xi)}+\tau^2||V||^2_{L^2(\Omega_\xi)}\nonumber\\
&&\leq \frac{c}{|\tau-\widehat{\tau}_k|^2}(||F||^2_{L^2(\Omega_\xi)}
+||G||^2_{V^{1/2}(S_\xi)}+\tau ||G||^2_{L^2(S_\xi)})
\end{eqnarray}
and
\begin{equation}\label{F9bax}
c_0C\sin(\kappa\log(\tau/\widehat{\tau}_k)+\psi(\tau)-\psi(\widehat{\tau}_k))=\int_{\Omega_\xi}\widehat{U}_\tau FdXdY+\int_{S_\xi}\widehat{U}_\tau Gds,
\end{equation}
where the function $\widehat{U}_\tau$ is introduced at the end of the previous section.
\end{lemma}

\begin{proof}
First let us prove (\ref{F9bax}). For small $\epsilon$ we intruduce $\Omega^\epsilon=\Omega_\xi\setminus B_\epsilon$ and
$S^\epsilon=\S_\xi\setminus V_\epsilon$, where $B_\epsilon$ is the disc of radius $\epsilon$ with the center at $(0,\xi(0))$. Then using Green's formula we get
\begin{eqnarray*}
&&\int_{\Omega^\epsilon}F\Phi_\tau dXdY+\int_{S^\epsilon}G\Phi_\tau ds
=C\Big(\partial_r(K_{i\kappa}(\tau r)+Q(\tau)I_{i\kappa}(\tau r))\sin(\kappa\log(r/2)+\gamma)\\
&&-(K_{i\kappa}(\tau r)+Q(\tau)I_{i\kappa}(\tau r))\partial_r\sin(\kappa\log(r/2)+\gamma)\Big)|_{R=\epsilon}.
\end{eqnarray*}
Using asymptotics (\ref{F23a}) and (\ref{F23aa}), we get
$$
\int_{\Omega_\xi}F\Phi_\tau dXdY+\int_{S_\xi}G\Phi_\tau ds=Cc_0\sin(\kappa\log\tau-\gamma_\kappa-\gamma+\psi(\tau)).
$$
By (\ref{Okt18b})
$$
\kappa\log\widehat{\tau}_k-\gamma_\kappa-\gamma+\psi(\widehat{\tau}_k)=k\pi,
$$
and we arrive at (\ref{F9bax}).

From (\ref{F9bax}) it follows that
\begin{equation}\label{F19aa}
|C|\leq \frac{c}{|\tau-\widehat{\tau}_k|}(||F||_{L^2(\Omega_\xi)}+\tau^{1/2}||G||_{L^2(S_\xi)}).
\end{equation}

Since the operator ${\mathcal L}_0$ is self-adjoint we have
$$
||U||_{L^2(\Omega_\xi)}\leq \frac{c}{\tau|\tau-\widehat{\tau}_k|}||F||_{L^2(\Omega_\xi)},\;\;\mbox{if $G=0$}.
$$
and due to (\ref{F14az}) we get
\begin{equation}\label{F19ab}
||V||_{L^2(\Omega_\xi)}\leq \frac{c}{\tau|\tau-\widehat{\tau}_k|}||F||_{L^2(\Omega_\xi)}\;\;\mbox{in the case $G=0$}.
\end{equation}

By (\ref{F14az}) we can write the equation for $V$ as
\begin{eqnarray}\label{F19ac}
&&{\mathcal L}_0V+\tau^2 V=f:=F-C[{\mathcal L}_0,\zeta(\tau R)\sin(\kappa\log \frac{1}{2}R+\gamma)]-C\tau^2\zeta(\tau R)\sin(\kappa\log\frac{1}{2}R+\gamma)\;\;\mbox{in $\Omega_\xi$}\nonumber\\
&&{\mathcal B}_0V=g=G-C[{\mathcal B}_0,\zeta(\tau R)\sin(\kappa\log\frac{1}{2}R+\gamma)]\;\;\mbox{on $S_\xi$}\nonumber\\
&&V(X,0)=0,\;\;V|_{x=0}=0,\;\;V|_{x=\Lambda/2}=0.
\end{eqnarray}
We take first $\tau=\tau_*$ in (\ref{F19ac}), where $\tau_*=\frac{1}{2}q\tau_k$, where $\tau_k$ was introduced in Theorem \ref{TM10}. We solve the above problem first for this value of $\tau$ and denote corresponding solution by $V=V_\tau$.

Let $\zeta_1(r)$ be a cut-off function $\zeta_1(r)=1$ for $r<\delta$ and $\zeta_1(r)=0$ for $r>2\delta$. We write
$$
V=V_1+V_2,\;\;V_1=\zeta_1V,\;V_2=\zeta_2V,\;\;\mbox{where $\zeta_2=(1-\zeta_1)$}.
$$
Then we get
\begin{eqnarray}\label{F19aca}
&&{\mathcal L}_0V_j+\tau^2 V_j=f_j\;\;\mbox{in $A$}\nonumber\\
&&{\mathcal B}_0V_j=g_j\;\;\mbox{on $(0,\infty)$}\nonumber\\
&&V(X,0)=0,\;\;\partial_XV|_{X=0}=0,\;\;\partial_XV|_{X=\Lambda/2}=0.
\end{eqnarray}
where
$$
f_j=\zeta_jf+2\partial_r(V\partial_r\zeta_j)-V\partial_r^2\zeta_j,\;\;g_j=\zeta_jg,\;\;j=1,2.
$$
Applying Proposition \ref{LF20a} to (\ref{F19aca}) with $j=1$, we obtain
\begin{equation}\label{F24ba}
||\zeta_1V||_{V^2(A)}+\tau^2||\zeta_1V||_{L^2(A)}\leq c(||\zeta_1f||_{L^2(A_{4\delta})}+||\zeta_1g||_{V^{1/2}}+\tau^{1/2}||\zeta_1g||_{L^2} +||V||_{L^2(A_{\delta,4\delta})})
\end{equation}
where $A_{s}=\{r<s,\;\theta\in (0,\alpha^*)\}$ and $A_{t,s}=\{t<r<s,\;\theta\in (0,\alpha^*)\}$.


Multiplying the first equation in (\ref{F19aca}) with $j=2$ by $\zeta_2V$ and integrating over $\Omega_\xi$, we get after integration by parts
$$
a(\zeta_2V,\zeta_2V)\leq c(\int_{\Omega_\xi}f_2\zeta_2 V+\int_0^{\Lambda/2} g_2\zeta_2Vdx.
$$
Using (\ref{F24b}) we get the estimate for  $\tau>C_\delta$
$$
\int_{\Omega_\xi}(|\nabla \zeta_2V|^2+\tau^2|\zeta_2V|^2)dXdY\leq C_0\tau^{-2}(||\zeta_2f||^2_{L^2}+\tau ||\zeta_2g||^2_{L^2})+C_\delta ||V||^2_{L^2(A_{\delta/2,3\delta})}.
$$
Now using local estimates with the parameter $\tau$ and partition of unity, we get
\begin{equation*}
||\nabla^2V||_{L^2(\Omega_\xi\setminus B_\delta)}\leq C_\delta\Big(||f||_{L^2(\Omega_\xi\setminus B_{\delta/2})}+||g||_{V^{1/2}(S_\xi\setminus B_{\delta/2})}+\tau^{1/2}||g||_{L^2(S_\xi\setminus B_{\delta/2})}\Big)+C_\delta ||V||_{L^2(A_{\delta/2,\infty})}
\end{equation*}
From this estimate and from (\ref{F24ba}) it follows the estimate (\ref{F9bx})  for $V=V_\tau$, $\tau=\tau_*$. The equation for $W=V_\tau-V_{\tau_*}$ is the following
\begin{eqnarray}\label{F19acaz}
&&{\mathcal L}_0W+\tau^2 W=(\tau_*^2-\tau^2)V_{\tau_*}+(C_\tau-C_{\tau_*})[{\mathcal L}_0,\zeta(r)\sin(\kappa\log r+\gamma)]\;\;\mbox{in $A$}\nonumber\\
&&{\mathcal B}_0W=0\;\;\mbox{on $(0,\infty)$}\nonumber\\
&&W(X,0)=0,\;\;\partial_XW|_{X=0}=0,\;\;\partial_XW|_{X=\Lambda/2}=0.
\end{eqnarray}
Using estimates (\ref{F19ab}) and (\ref{F19aa}), we obtain
$$
||W||_{L^2(\Omega_\xi)}\leq C\frac{1}{|\tau-\tau_k|}||V_{\tau_*}||_{L^2(\Omega_\xi)},
$$
which completes the proof.
\end{proof}

Let
\begin{eqnarray}\label{F23ab}
&&{\mathcal D}_\gamma^\beta=\{U=C\zeta(\tau R)\sin(\kappa\log \frac{1}{2}R+\gamma)+V\,:\, V\in \tilde{V}^2_\beta(\Omega_\xi),\,\;\mbox{$C$ is a constant}\nonumber\\
&&\mbox{and $V$ satisfies (\ref{F8aa})}.
\end{eqnarray}

Let also
$$
\widehat{n}(\tau)=\frac{\sin(\kappa\log(\tau/\widehat{\tau}_k)+\psi(\tau)-\psi(\widehat{\tau}_k))}{\tau/\widehat{\tau}_k-1}
$$
and
$$
\widehat{m}(X,Y,\tau)=\frac{U\tau-U_{\widehat{\tau}_k}}{\tau/\widehat{\tau}_k-1}.
$$

\begin{theorem}\label{TF9aa} Let $0\leq\beta<1$, $q_0\in (0,q)$ and
$$
\tau\in [q_0^{-1}\widehat{\tau}_k,q_0\widehat{\tau}_k],
$$
where that $k\geq k_\delta$ and $k_\delta$ is a sufficiently large integer depending on $\delta$.
Assume that $F\in L^2_\beta(\Omega_\xi)$ and $G\in V^{1/2}_\beta(S_\xi)$ satisfy
\begin{equation}\label{F25a}
\int_{\Omega_\xi}F\widehat{U}_kdXdY+\int_{S_\xi}G\widehat{U}_kds=0.
\end{equation}

Then there exists a unique solution $U\in {\mathcal D}_\gamma^\beta$ of the form (\ref{F23ab})
solving {\rm (\ref{F8a})}, {\rm (\ref{F8aa})}, satisfying
$$
\int_{\Omega_\xi}U\widehat{U}_kdXdY=0
$$
 and
\begin{eqnarray}\label{F9b}
&&\tau^{-\beta}|C|+\tau^{-1}||V||_{V^2_\beta(\Omega_\xi)}+\tau||V||_{L^2_\beta(\Omega_\xi)}\nonumber\\
&&\leq c\tau^{-1}(||F||_{L^2_\beta(\Omega_\xi)}
+||G||_{V^{1/2}_\beta(S_\xi)}+\tau^{1/2} ||G||_{L^2_\beta(S_\xi)}),
\end{eqnarray}
where $c$ does not depend on $\delta$. Moreover
\begin{equation}\label{F24bax}
c_0C_1\widehat{n}(\tau)=\int_{\Omega_\xi}\widehat{m}(X,Y,\tau) FdXdY+\int_{S_\xi}\widehat{m}(X,Y,\tau) Gds,
\end{equation}
\end{theorem}
\begin{proof} We start from the estimate of the constant $C$ in (\ref{F23ab}).
\begin{equation}\label{F25aa}
|C
|\leq \frac{c}{|\tau|^{\beta-1}}(||F||_{L^2(\Omega_\xi)}+\tau^{1/2}||G||_{L^2(S_\xi)}).
\end{equation}

Since the operator is self-adjoint and due to (\ref{F25a}) we have
$$
||U||_{L^2(\Omega_\xi)}\leq \frac{c}{\tau|\tau-\widehat{\tau}_k|}||F||_{L^2(\Omega_\xi)},\;\;\mbox{if $G=0$}.
$$
and due to (\ref{F14az}) we get
\begin{equation}\label{F19abz}
||V||_{L^2(\Omega_\xi)}\leq \frac{c}{\tau^2}||F||_{L^2(\Omega_\xi)}\;\;\mbox{in the case $G=0$}.
\end{equation}

Now consider the case $\beta=0$.  We take the same $\tau_*$ as in that lemma and construct solution $U_{\tau_*}$, which satisfies the estimate (\ref{F9b}).
The function $W=V_\tau-V_{\tau_*}$ satisfies the problem
\begin{eqnarray}\label{F26a}
&&L_0W+\tau^2 W=(\tau_*^2-\tau^2)V_{\tau_*}+(C_\tau-C_{\tau_*})[L_0,\zeta(r)\sin(\kappa\log r+\gamma)]\;\;\mbox{in $A$}\nonumber\\
&&B_0W=0\;\;\mbox{on $(0,\infty)$}\nonumber\\
&&W(X,0)=0,\;\;W|_{x=0}=0,\;\;W|_{x=\Lambda/2}=0.
\end{eqnarray}
Using estimates (\ref{F25aa}), (\ref{F19abz}) and (\ref{F9bx}), for $\tau=\tau_*$, we arrive at (\ref{F9b}) for $\beta=0$.

The case of  $\beta\in (0,1)$ is obtained by using first a local estimate near the vertex $(0,\xi(0))$.
We represent solution $U\in {\mathcal D}_\gamma^\beta$ as $U=U_0+W_\beta$, where $W_\beta=\zeta U$ solves the problem
\begin{eqnarray}\label{F27a}
&&L_0W_\beta+\tau^2 W_\beta=\zeta f+[L_0,\zeta]U\;\;\mbox{in $A$}\nonumber\\
&&B_0W_\beta=\zeta g\;\;\mbox{on $(0,\infty)$}\nonumber\\
&&W_\beta(X,0)=0,\;\;W_\beta|_{x=0}=0,\;\;W_\beta|_{x=\Lambda/2}=0.
\end{eqnarray}
By Proposition \ref{PF22}(i) this problem has a solution and it satisfies
$$
||W_\beta||_{V^2_\beta(A)}\leq c(\tau^2||W_\beta||_{L^2_\beta(A)}+||[L_0,\zeta]U||_{L^2_\beta(A)}+||\zeta f||_{L^2_{\beta}(A)}+||\zeta g||_{L^2_{\beta}(0,\infty)})
$$
Then for $U_0$ we obtain the problem
\begin{eqnarray}\label{F27aa}
&&L_0U_0+\tau^2 U_0=(1-\zeta )f+[L_0,\zeta]U\;\;\mbox{in $\Omega_\xi$}\nonumber\\
&&B_0U_0=(1-\zeta) g\;\;\mbox{on $S_\xi$}\nonumber\\
&&U_0(X,0)=0,\;\;U_0|_{X=0}=0,\;\;U_0|_{X=\Lambda/2}=0.
\end{eqnarray}
Since the right-hand side here belongs to $L^2(\Omega_\xi)$ and $L^2(S_\xi)$ respectively we can apply the estimate (\ref{F9b}) for $\beta=0$ which we have proved already. Using we assume  $\beta>0$ this is sufficient to complete the proof.
\end{proof}

\subsection{Spectral problem for perturbed operator}\label{Spert}

 Let $\beta\in (1-\alpha,1)$. Introduce
 $$
 X^\beta_\gamma=\{U=C\zeta w_\gamma+v\,:\, v\in V^{2}_\beta(\Omega),\; v\;\mbox{satisfies  (\ref{F8aaa})}\},
 $$
 where
 $$
 w_\gamma =\sin(\kappa\log \frac{1}{2}r+\gamma)\phi_0\;\;\mbox{and $C$ is constant}.
 $$

 We consider in this section operators $({\mathcal L}_0,{\mathcal B}_0)$ and $({\mathcal L},{\mathcal B})$ with the same domain $X^\beta_\gamma$, $\beta\in (1-\alpha,1)$. By using local estimates (see Proposition \ref{P2}) one can check that the corresponding eigenfunctions belongs to ${\mathcal D}_\gamma$. So it is enough to perform the proof for the domain $X^\beta_\gamma$.


In Sect.\ref{SM7aa} we have found large negative eigenvalues  $\widehat{\lambda}_k=-\widehat{\tau}_k^2$, $k\geq k_\delta$, of the unperturbed operator and corresponding normalized eigenfunctions ${\bf U}_k$
$$
{\mathcal L}_0{\bf U}_k+\widehat{\tau}_k^2{\bf U}_k,\;\;{\mathcal B}_0{\bf U}_k =0.
$$
We are looking for a solution to
$$
{\mathcal L}U=\tilde{\lambda} U,\;\;{\mathcal B}U =0
$$
from the space $ X^\beta_\gamma$ in the form
\begin{equation}\label{M1aa}
U={\bf U}_k+V,\;\;\int_{\Omega_\xi} {\bf U}_k VdXdY=0,
\end{equation}
where
\begin{equation}\label{M1a}
V=C\zeta(\tau r)w_\gamma+v,\;\;v\in\tilde{V}^2_\beta(\Omega_\xi).
\end{equation}
Therefore,
\begin{equation}\label{D11a}
{\mathcal L}({\bf U}_k+V)=\tilde{\lambda} ({\bf U}_k+V),\;\;{\mathcal B}({\bf U}_k+V)=0.
\end{equation}

First we consider (\ref{D11a}) as equation with respect to $V$:
\begin{equation}\label{D11ac}
({\mathcal L}_0-\tilde{\lambda})V=(\tilde{\lambda}-\widehat{\lambda}_k){\bf U}_k-{\mathcal L}_1({\bf U}_k+V),\;\;{\mathcal B}_0V=-{\mathcal B}_1({\bf U}_k+V).
\end{equation}
According to Proposition \ref{TF9aa}  for solvability od this problem with respect to $V$  we must require
\begin{equation}\label{F10a}
\Big((\tilde{\lambda}-\widehat{\lambda}_k){\bf U}_k-{\mathcal L}_1({\bf U}_k+V),{\bf U}_k\Big)_{\Omega_\xi}-\Big({\mathcal B}_1({\bf U}_k+V),{\bf U}_k\Big)_{S_\xi}=0,
\end{equation}
where $(\cdot,\cdot)_{\Omega_\xi}$ and $(\cdot,\cdot)_{S_\xi}$ are inner products in $L^2(\Omega_\xi)$ and $L^2(S_\xi)$ respectively.
To guarantee the solvability condition for (\ref{D11ac}) we replace the right hand side in (\ref{D11ac})  as follows
\begin{eqnarray}\label{D11bb}
&&({\mathcal L}_0-\tilde{\lambda})V={\mathcal N}V+{\mathcal F}\nonumber\\
&&{\mathcal B}_0V=-{\mathcal B}_1(\widehat{U}_k+V).
\end{eqnarray}
where
$$
{\mathcal N}V=-{\mathcal L}_1V+\Big((V,{\mathcal L}_1{\bf U}_k)_{\Omega_\xi}+(V,{\mathcal B}_1{\bf U}_k)_{S_\xi}\Big){\bf U}_k
$$
and
\begin{eqnarray*}
&&{\mathcal F}=(\widetilde{\lambda}-\widehat{\lambda}_k) {\bf U}_k-{\mathcal L}_1{\bf U}_k+\Big({\bf U}_k,{\mathcal L}_1{\bf U}_k)_{\Omega_\xi}+({\bf U}_k,{\mathcal B}_1{\bf U}_k)_{S_\xi}-(\tilde{\lambda}-\widehat{\lambda}_k) ({\bf U}_k,{\bf U}_k)\Big){\bf U}_k\\
&&=-{\mathcal L}_1{\bf U}_k+\Big({\bf U}_k,{\mathcal L}_1{\bf U}_k)_{\Omega_\xi}+({\bf U}_k,{\mathcal B}_1{\bf U}_k)_{S_\xi}\Big){\bf U}_k
\end{eqnarray*}
Now the right hand side of (\ref{D11bb}) satisfies (\ref{F25a}) and according to Theorem \ref{TF9aa} the function $V$ exists and satisfies the estimate
\begin{eqnarray}\label{M3a}
&&\tau^{-\beta}|C|+\tau^{-1}||v||_{V^2_\beta(\Omega_\xi)}+\tau||v||_{L^2_\beta(\Omega_\xi)}\nonumber\\
&&\leq c\tau^{-1}(||{\mathcal N}V+{\mathcal F}||_{L^2_\beta(\Omega_\xi)}
+||{\mathcal B}_1({\bf U}_k+V)||_{V^{1/2}_\beta(S_\xi)}+\tau^{1/2} ||{\mathcal B}_1({\bf U}_k+V)||_{L^2_\beta(S_\xi)}).
\end{eqnarray}
Using (\ref{M1a}) together with {F16aa}--(\ref{F22ac}), we get
$$
||{\mathcal N}V||_{L^2_\beta(\Omega_\xi)}\leq c\Big(|C|\tau^{-\alpha-\beta}+\delta^\alpha ||v||_{V^2_\beta(\Omega_\xi)}+C_\delta (||\nabla v||_{L^2_\beta(\Omega_\xi)}+||V||_{L^2_\beta(\Omega_\xi)})\Big)
$$
and
\begin{eqnarray*}
&&||{\mathcal B}_1V||_{V^{1/2}_\beta(S_\xi)}+\tau^{1/2} ||{\mathcal B}_1V||_{L^2_\beta(S_\xi)}\leq c\Big(|C|\tau^{-\alpha-\beta}+\delta^\alpha(||v||_{V^{2}_\beta(\Omega_\xi)}+\tau^{1/2} ||v||_{V^1_\beta(\Omega_\xi)})\\
&&+C_\delta (||v||_{L^2_\beta(\Omega_\xi)}+\tau^{1/2} ||v||_{L^2_\beta(\Omega_\xi)})\Big).
\end{eqnarray*}
The last two estimates applied to (\ref{M3a}) imply
\begin{eqnarray}\label{M3aa}
&&\tau^{-\beta}|C|+\tau^{-1}||v||_{V^2_\beta(\Omega_\xi)}+\tau||v||_{L^2_\beta(\Omega_\xi)}\nonumber\\
&&\leq c\tau^{-1}(||{\mathcal F}||_{L^2_\beta(\Omega_\xi)}
+||{\mathcal B}_1{\bf U}_k||_{V^{1/2}_\beta(S_\xi)}+\tau^{1/2} ||{\mathcal B}_1{\bf U}_k||_{L^2_\beta(S_\xi)})
\end{eqnarray}
for large $\tau$.
Using again (\ref{M1a}), we get
\begin{eqnarray}\label{M3aaa}
&&\tau^{-\beta}|C|+\tau^{-1}||v||_{V^2_\beta(\Omega_\xi)}+\tau||v||_{L^2_\beta(\Omega_\xi)}\nonumber\\
&&\leq c\tau^{-1}(||{\mathcal F}||_{L^2_\beta(\Omega_\xi)}
+||{\mathcal B}_1{\bf U}_k||_{V^{1/2}_\beta(S_\xi)}+\tau^{1/2} ||{\mathcal B}_1{\bf U}_k||_{L^2_\beta(S_\xi)}).
\end{eqnarray}

Now using estimates (\ref{M3c1})-(\ref{M3cb1}), we obtain
\begin{equation}\label{M7b}
||{\mathcal F}||_{L^2_\beta(\Omega_\xi)}\leq c\widehat{\tau_k}^{2-\beta-\alpha}
\end{equation}
and
\begin{equation}\label{M7ba}
||{\mathcal B}_1{\bf U}_k||_{V^{1/2}_\beta(S_\xi)}+\tau^{1/2} ||{\mathcal B}_1{\bf U}_k||_{L^2_\beta(S_\xi)})\leq c\widehat{\tau_k}^{2-\beta-\alpha}.
\end{equation}
Therefore
\begin{equation}\label{Apr26a}
\tau^{-\beta}|C|+\tau^{-1}||v||_{V^2_\beta(\Omega_\xi)}+\tau||v||_{L^2_\beta(\Omega_\xi)}\leq  c\widehat{\tau_k}^{1-\beta-\alpha}.
\end{equation}

Equation for $\widetilde{\lambda}$ is (\ref{F10a}) which is
\begin{equation}\label{D11bag}
(\tilde{\lambda}-\widehat{\lambda}_k)=\Big({\mathcal L}_1({\bf U}_k+V),{\bf U}_k\Big)_{\Omega_\xi}+\Big({\mathcal B}_1({\bf U}_k+V),{\bf U}_k\Big)_{S_\xi}.
\end{equation}

Furthermore, by (\ref{M3aa}), (\ref{M7b}) and (\ref{M7ba})
$$
|\Big({\mathcal L}_1({\bf U}_k+V),{\bf U}_k\Big)_{\Omega_\xi}|\leq C ||{\mathcal L}_1({\bf U}_k+V)||_{L^2_{\beta}(\Omega_\xi)}
||{\bf U}_k||_{L^2_{-\beta}(\Omega_\xi)}.
$$
Since $\beta<1$ we have by  (\ref{M3c1})
$$
||{\bf U}_k||_{L^2_{-\beta}(\Omega_\xi)}\leq c\widehat{\tau}_k^{\beta}.
$$
Using that $\beta<1-\alpha$ we obtain by (\ref{M3cb1}) and (\ref{Apr26a})
$$
||{\mathcal L}_1({\bf U}_k+V)||_{L^2_{\beta}(\Omega_\xi)}\leq \widehat{\tau}_k^{2-\beta-\alpha}.
$$
Hence
$$
|\Big({\mathcal L}_1({\bf U}_k+V),{\bf U}_k\Big)_{\Omega_\xi}|\leq c\widehat{\tau}_k^{2-\alpha}.
$$

Now let us turn to the boundary term in the right-hand side of (\ref{D11bag}). We have
$$
|\Big({\mathcal B}_1({\bf U}_k+V),{\bf U}_k\Big)_{S_\xi}|\leq c||{\mathcal B}_1({\bf U}_k+V)||_{L^2_{\beta-1/2}(S_\xi)}||{\bf U}_k||_{L^2_{-\beta+1/2}(S_\xi)}.
$$
Using definition of ${\bf U}_k$ in Sect.\ref{SM7aa}, we get
$$
||{\bf U}_k||_{L^2_{-\beta+1/2}(S_\xi)}\leq c\widehat{\tau}_k^\beta
$$
Since
$$
||\nabla v||_{L^2_{\beta-1/2}(S_\xi)}\leq c||v||_{V^2_\beta(\Omega_\xi)}
$$
and
$$
||v||_{L^2_{\beta-3/2}(S_\xi)}\leq c||v||_{V^1_{\beta-1}(\Omega_\xi)},
$$
we get by means of (\ref{Apr26a}) and (\ref{M3c1})--(\ref{M3cb1})
$$
|\Big({\mathcal B}_1({\bf U}_k+V),{\bf U}_k\Big)_{S_\xi}|\leq c\widehat{\tau}_k^{2-\beta-\alpha}.
$$
Thus the right-hand side of (\ref{D11bag}) is estimated as $O(\tau^{2-\alpha})$ which leads to the formula
$$
\widetilde{\lambda}-\widehat{\lambda}_k=O(\widehat{\tau}_k^{2-\alpha})
$$
The last relation gives (1.14) if we use that the right hand side in (4.66) is continuous with respect to $t$ and $\widetilde{\lambda}$ and that we can apply the same procedure to the perturbation $({\mathcal L}_0+t{\mathcal L}_1,{\mathcal B}_0+t{\mathcal B}_1)$, $t\in [0,1]$, and observe that the eigenvalue cannot leave the interval
$$
\widetilde{\lambda}\in [\widehat{\lambda}_k-c\widehat{\tau}_k^{2-\alpha},\widehat{\lambda}_k+c\widehat{\tau}_k^{2-\alpha}].
$$

\begin{remark}\label{RM1} In the above proof we obtain also the  asymptotic formula (\ref{M1aa}) for the eigenfunction corresponding  to the eigenvalue $-s_k^2$. The function $V$ having the representation (\ref{M1a})  can be considered as a remainder and is estimated as
\begin{equation}\label{M1b}
s_k^{-\beta}|C|+s_k^{-1}||v||_{V^2_\beta(\Omega_\xi)}+s_k||v||_{L^2_\beta(\Omega_\xi)}\leq cs_k^{1-\alpha-\beta}
\end{equation}
(see (\ref{Apr26a})). Here as before
 $\beta\in (1-\alpha,1)$. The estimate (\ref{M1b}) implies the following estimates for the constant $C$ and $L^2$ norm of $v$ in the representation  (\ref{M1a})
\begin{equation}\label{M1ba}
|C|\leq cs_k^{1-\alpha}\;\;\;\mbox{and}\;\;\;||v||_{L^2(\Omega_\xi)}\leq cs_k^{-\alpha}.
\end{equation}
Finally, we get
\begin{equation}\label{M1bad}
||V||_{L^2(\Omega_\xi)}\leq cs_k^{-\alpha}.
\end{equation}
\end{remark}

\section{Appendix}

\subsection{Derivation of the spectral problem for water waves of extreme form}\label{SecWater}

In this section we will derive the linear system that arises as a linearization of the water wave problem near an extreme Stokes wave. The stream function formulation for steady waves on the free surface of rotational flows of finite depth is given by

\begin{subequations}\label{K2a}
	\begin{alignat}{2}
		\Delta\psi+\omega(\psi)&=0 &\qquad& \text{in } \hat{D},\label{psilap} \\
		\tfrac 12|\nabla\psi|^2 + y  &= R &\quad& \text{on } \hat{S},\label{psibern}\\
		\psi  &= m &\quad& \text{on }\hat{S},\label{psitop}\\
		\psi  &= 0 &\quad& \text{on }\hat{B}.\label{psibot}
	\end{alignat}
\end{subequations}

Here $\psi$ is the stream function, $\omega$ is the vorticity, $m > 0$ is the relative mass flux and $R$ is the Bernoulli constant. The unknown region $\hat{D}$ is defined as
\[
	\hat{D} = \{ (x,y) \in \R^2: \ \ 0< y < \eta(x) \},	
\]
while $\hat{S}$ and $\hat{D}$ stand for the upper and lower boundaries respectively. For more details about the derivation of \eqref{K2a} we refer to the book \cite{Constantin11b}.

Throughout this section we will assume that $(\psi,\eta)$ is an extreme Stokes wave solution, that is
\[
	\eta(0) = R, \ \ \psi_x (0,R) = \psi_y (0,R) = 0,
\]
and
\[
	\lim_{x\to 0+} \eta'(x)=-\frac{1}{\sqrt{3}}.
\]
The solution is even in $x$-variable and has period $\Lambda$. As for the regularity, we initially have $\psi \in C^1(\overline{\hat{D}})$ and $\eta \in C(\R)$. However, outside the stagnation points the regularity is better as recently shown in \cite{Koz2020asymp}. More precisely, we can assume that
\begin{equation}\label{etax}
	\eta_x = -\frac{1}{\sqrt{3}} + a_1 x^{\tfrac12} + a_2 x + f_1(x),
\end{equation}
where $f_1 \in C^{1}(\overline{\Omega})$, $f_1(x) = O(|x|^{3/2(\tau_1-1)})$, $f'_1(x) = O(|x|^{3/2(\tau_1-1)-1})$ and $f''_1(x) = O(|x|^{3/2(\tau_1-1)-2})$ as $x \to 0$. Here $\tau_1 \approx 1.8$ is the smallest root of $\tau_1 = - \tfrac{1}{\sqrt{3}} \cot(\tfrac{\pi}{2} \tau_1)$, while $\Omega = \{(x,y) \in \hat{D}: 0 < x < \Lambda/2\}$. In order to specify the regularity of $\psi$ it is convenient to introduce polar coordinates
\begin{equation}\label{D15a}
	x=r\sin\theta,\;\;R-y=r\cos\theta,
\end{equation}
where $\theta = 0$ corresponds to the vertical line $x=0$. Therefore, we have $\theta \to \pi/3$ along the surface as $x \to 0+$. Thus, the corresponding representation of $\psi$ is
\begin{equation}\label{Dec11a}
	\psi(x,y)=m-\tfrac{2}{3}r^{3/2}\cos \Big(\tfrac32 \theta \Big)+f_2(x,y),
\end{equation}
where $f_2 \in C^2(\overline{\Omega})$ and $f_2(x,y) = O(r^2)$ as $r \to 0$.

Now we can formally take variations in \eqref{K2a} with respect to $\psi$ and $\eta$. Thus, if $u$ and $\zeta$ are the corresponding variations of $\psi$ and $\eta$ respectively, then the linear spectral problem associated with \eqref{K2a} is
\begin{subequations}\label{K2a}
	\begin{alignat}{2}
		\Delta u+\omega'(\psi)u&=\lambda u &\qquad& \text{in } \Omega,\label{ulapp} \\
		\nabla \psi \cdot \nabla u + \left( 1 + \psi_x \psi_{xy} + \psi_y \psi_{yy} \right) \zeta  &= 0 &\quad& \text{on } y = \eta,\label{ubern}\\
		u + \psi_y \zeta   &= 0 &\quad& \text{on }y = \eta,\label{utop}\\
		u  &= 0 &\quad& \text{on } y=0. \label{ubot}
	\end{alignat}
\end{subequations}
Note that we can express $\zeta = - u / \psi_y$ from \eqref{utop}, so that the boundary relation \eqref{ubern} becomes
\[
\nabla \psi \cdot \nabla u - \frac{u}{\psi_y} \left( 1 + \psi_x \psi_{xy} + \psi_y \psi_{yy} \right) = 0.
\]
Taking into account equations for $\psi$, we can rewrite this equation as
\[
\partial_{\nu} u - r^{-1} \rho u = 0,
\]
where
\[
	\rho = r \frac{1-\omega(1)(1+\eta'^2)\psi_y(x,\eta(x)) + \eta'' \psi_y^2(x,\eta(x))}{2(R-\eta)\sqrt{1+\eta'^2} }.
\]
Now it follows from \eqref{etax} and \eqref{Dec11a} that
\[
\rho = \frac{\sqrt{3}}2 + O(x^{\tfrac12}), \ \ \rho' = O(x^{-\tfrac12})  \ \ \text{as} \ \ x \to 0.
\]
Thus, the smallest positive root $\mu_1$ to
\[
	\mu \tan(\tfrac{\pi}{3} \mu) = - \tfrac{\sqrt{3}}{2}
\]
equals to $(3/2) \tau_1$, where $\tau_1 \approx 1.8$. We see that all assumptions of Theorem 1.1 are fulfilled for the system
\begin{subequations}\label{sysu}
	\begin{alignat}{2}
		\Delta u+\omega'(\psi)u&=\lambda u &\qquad& \text{in } \Omega,\label{ulapp1} \\
		\partial_{\nu} u - r^{-1} \rho u  &= 0 &\quad& \text{on } y=\eta,\label{ubern1}\\
		u_x &= 0 &\quad& \text{on } x=0, \ x = \tfrac12\Lambda,\label{ux1}\\
		u  &= 0 &\quad& \text{on } y=0, \label{ubot1}
	\end{alignat}
\end{subequations}
where $\alpha^* = \pi/3$, $\alpha = 1/2$, $a_0 = 1/\sqrt{3}$ and $\rho_0 = \frac{\sqrt{3}}2$. Furthermore, the constant $\kappa \approx 1.07$ is defined as a solution to \eqref{K49zq}.

\bigskip
\noindent {\bf Acknowledgements.} V.~K. was supported by the Swedish Research
Council (VR), 2017-03837.

\bibliographystyle{siam}
\bibliography{bibliography}

\end{document}